\documentclass[10pt,reqno]{amsart}
\usepackage{graphicx}
\usepackage{amsfonts}
\usepackage{amssymb}
\usepackage{amsmath}
\usepackage{amsxtra}
\usepackage{latexsym}
\usepackage{epstopdf}
\usepackage{mathrsfs}
\usepackage{mathtools}
\usepackage{esint}
\usepackage{graphicx, caption, subcaption}
\usepackage{enumitem}

\newtheorem{theorem}{Theorem}[section]
\newtheorem{lemma}[theorem]{Lemma}

\newtheorem{algorithm}{Algorithm}

\theoremstyle{definition}

\newtheorem{example}{Example}[section]

\theoremstyle{remark}
\newtheorem{remark}{Remark}[section]

\numberwithin{equation}{section}

\begin{document}

\title[Stochastic heavy-ball method ]{Convergence analysis of a stochastic 
heavy-ball method for linear ill-posed problems} 

\author{Qinian Jin}
\address{Mathematical Sciences Institute, Australian National
University, Canberra, ACT 2601, Australia}
\email{qinian.jin@anu.edu.au} \curraddr{}

\author{Yanjun Liu}
\address{Mathematical Sciences Institute, Australian National
University, Canberra, ACT 2601, Australia}
\curraddr{Department of Operations Research and Financial Engineering,
Princeton University, Princeton, NJ 08544, USA}
\email{yanjun.liu@princeton.edu}




\begin{abstract}
In this paper we consider a stochastic heavy-ball method for solving linear ill-posed inverse problems. 
With suitable choices of the step-sizes and the momentum coefficients, we establish the regularization 
property of the method under {\it a priori} selection of the stopping index and derive the rate of 
convergence under a benchmark source condition on the sought solution. Numerical results are provided to 
test the performance of the method. 
\end{abstract}

%
%

\def\d{\delta}
\def\E{\mathbb{E}}
\def\P{\mathbb{P}}
\def\X{\mathcal{X}}
\def\A{\mathcal{A}}
\def\Y{\mathcal{Y}}
\def\l{\langle}
\def\r{\rangle}
\def\by{\overline{y}^{(n)}}
\def\bY{\overline{y}_n}
\def\la{\lambda}
\def\EE{{\mathbb E}}
\def\RR{{\mathbb R}}
\def\a{\alpha}
\def\l{\langle}
\def\r{\rangle}
\def\p{\partial}
\def\ep{\varepsilon}
\def\N{\mathcal N}
\def\F{\mathcal F}
\def\R{\mathcal R}
\def\prox{\mbox{prox}}

%
%

\newcommand{\ylnote}[1]{\textbf{{\color{blue}{Yanjun:}   #1 }}}

\maketitle

\section{\bf Introduction}

In this paper we consider linear ill-posed inverse problems governed by the system  
\begin{align}\label{sys}
A_i x = y_i, \quad i = 1, \cdots, p
\end{align}
consisting of $p$ linear equations, where, for each $i = 1, \cdots, p$,  $A_i: X \to Y_i$ 
is a bounded linear operator between two real Hilbert spaces $X$ and $Y_i$. Such inverse problems
arise in many practical applications, including, for instance, 

\begin{enumerate}[leftmargin = 0.8cm]
\item[$\bullet$] linear inverse problems with discrete data; see \cite{BDP1985}.

\item[$\bullet$] reconstruction problems arising in various tomography imaging, such as computed tomography and photoacoustic/thermoacoustic tomography, see \cite{BBGHP2007,JW2013,N2001}.


\item[$\bullet$] discretizing infinite-dimensional linear inverse problems by projection 
or collocation methods leads to large scale linear system of the form (\ref{sys}); 
see \cite{EHN1996}.
\end{enumerate}

Throughout the paper we assume (\ref{sys}) has a solution. Since (\ref{sys}) may have 
many solutions, we will determine the $x_0$-minimal norm solution $x^\dag$, i.e. the 
solution with the shortest distance to $x_0$, where $x_0\in X$ is an initial guess. Let 
$Y:= Y_1 \times \cdots \times Y_p$ be the product space of $Y_1, \cdots, Y_p$ with the 
natural inner product inherited from those of $Y_i$. By setting $y := (y_1, \cdots, y_p)$ 
and introducing the bounded linear operator $A: X \to Y$ by 
$$
A x := (A_1 x, \cdots, A_p x), \quad x \in X,
$$
we may rewrite (\ref{sys}) equivalently as 
\begin{align}\label{sys2}
A x = y. 
\end{align}
In practical applications, data are usually acquired by measurements. Therefore, instead of 
the exact data $y = (y_1, \cdots, y_p)$, we have only noisy data $y^\d:= (y_1^\d, \cdots, y_p^\d)$ 
satisfying 
\begin{align}\label{noise}
\|y_i^\d - y_i\| \le \d_i, \quad i = 1, \cdots, p, 
\end{align}
where $\d_i>0$ denotes the noise level corresponding to the data in the space $Y_i$. Let 
$$
\d := \sqrt{\d_1^2 + \cdots +\d_p^2}
$$
which denotes the total noise level of the noisy data. It is thus important to design algorithm 
to produce stable approximation of $x^\dag$ using only the noisy data $y^\d$. Due to the 
ill-posedness of the underlying problems, regularization techniques are required to achieve 
the goal. Many regularization methods have been developed to solve ill-posed inverse problems; 
see \cite{EHN1996} and the references therein.

The most prominent iterative regularization method for solving (\ref{sys2}) is the Landweber iteration 
\begin{align}\label{LW}
x_{n+1}^\d = x_n^\d - \eta A^*(A x_n^\d - y^\d)
\end{align}
which can be viewed as the gradient method with constant step size $\eta>0$ applied to the quadratic function 
\begin{align}\label{quad}
f(x) = \frac{1}{2} \|A x - y^\d\|^2, \quad x \in X, 
\end{align}
where $A^*: Y \to X$ denotes the adjoint of $A$ given by 
$$
A^* v = \sum_{i=1}^p A_i^* v_i, \quad \forall v :=(v_1, \cdots, v_p) \in Y
$$
in terms of the adjoint $A_i^*$ of $A_i$ for all $i = 1, \cdots, p$. The regularization property of 
Landweber iteration has been carried out in detail in the literature, see \cite{EHN1996}. It is known 
that Landweber iteration is a slowly convergent method. 

In order to accelerate the Landweber iteration, multiple-step iterations have been 
considered to modify the method. In particular, based on using orthogonal polynomials, 
the two-step methods of the form 
\begin{align}\label{twostep}
x_{n+1}^\d = x_n^\d - \a_n \eta A^*(A x_n^\d - y^\d) + \beta_n (x_n^\d - x_{n-1}^\d) 
\end{align}
have been investigated in \cite{H1991} for $0<\eta<1/\|A\|^2$ and suitable choices of $\a_n$ and $\beta_n$. 
On the other hand, to determine a minimizer of the minimization problem 
\begin{align}\label{MP}
\min_{x\in X} f(x)
\end{align}
with a continuous differentiable objective function $f$, the heavy ball method 
\begin{align}\label{HB0}
x_{n+1} = x_n - \a_n \eta \nabla f(x_n) + \beta_n (x_n - x_{n-1})
\end{align}
has been proposed in \cite{P1964} to speed up the usual gradient descent method. The 
performance of (\ref{HB0}) depends crucially on the property of $f$ and the choices of 
$\a_n$ and $\beta_n$, and its analysis turns out to be very challenging. When $f$ is 
twice continuous differentiable, strongly convex and has Lipschitz continuous gradient, 
it has been demonstrated in \cite{P1964}, under suitable constant choices of $\a_n$ and 
$\beta_n$, that the iterative sequence $\{x_n\}$ enjoys a provable linear convergence 
faster than the gradient descent.  In the recent paper \cite{GFJ2015}, for merely convex 
objective functions, a convergence rate has been derived in terms of the objective 
function value under the choices $\a_n = 1/(n+2)$ and $\beta_n = n/(n+2)$ with a suitable 
small constant $\eta>0$. Note that, for the quadratic function $f$ given in (\ref{quad}), 
the method (\ref{HB0}) becomes the method (\ref{twostep}) and thus (\ref{twostep}) is a special 
case of (\ref{HB0}). However, the convergence results obtained in \cite{GFJ2015,P1964} are not applicable 
straightforwardly to (\ref{twostep}) when used to solve ill-posed inverse problems. Indeed, if the results 
in \cite{GFJ2015,P1964} were applicable, one would obtain a solution of the linear least square problem 
\begin{align}\label{LS}
\min_{x \in X} \|A x - y^\d\|^2
\end{align}
that generally may not have a solution unless $y^\d \in \mbox{Ran}(A) \oplus \mbox{Ran}(A)^\perp$, see 
\cite{EHN1996}, which is rarely satisfied because $y^\d$ contains irregular noise and $\mbox{Ran}(A)$ 
is non-closed. Even if (\ref{LS}) has a solution, this solution could be far away from the soluton of 
(\ref{sys2}) because of the ill-posedness. 

Note that the implementation of the Landweber iteration (\ref{LW}) and the heavy ball method 
(\ref{twostep}) requires to calculate 
\begin{align}\label{SL}
A^*(A x_n^\d - y^\d) = \sum_{i=1}^p A_i^* (A_i x_n^\d - y_i^\d)
\end{align}
at each iteration step. In case $p$ is huge, this calculation can be time-consuming. In order 
to reduce the computational load per iteration, we may consider using one random term selected 
from the right hand side of (\ref{SL}) to carry out the iteration step. Applying this 
randomization strategy to the Landweber iteration (\ref{LW}) leads to the stochastic gradient 
descent method for solving ill-posed problems, see \cite{JL2019,JLZ2023,LM2022} for instance. 
Applying the same randomization strategy to (\ref{twostep}) gives the following stochastic
heavy-ball method
\begin{align}\label{SHBM}
x_{n+1}^\d = x_n^\d - \a_n \eta_{i_n} A_{i_n}^* (A_{i_n} x_n^\d - y_{i_n}^\d) + \beta_n (x_n^\d - x_{n-1}^\d)
\end{align}
we will consider in this paper, where $x_{-1}^\d = x_0^\d := x_0\in X$ is the initial guess, 
$i_n \in \{1, \cdots, p\}$ is randomly chosen via the uniform distribution, and $\eta_1, \cdots, \eta_p$ 
are preassigned step sizes. This method differs from the stochastic gradient descent method in that 
a momentum term $\beta_n (x_n^\d - x_{n-1}^\d)$ is added to define iterates. The appearance of this 
momentum term brings challenging issues for analyzing the method. However, suitable choices of the 
step-size coefficient $\a_n$ and the momentum coefficient $\beta_n$ may accelerate the convergence 
speed. How to choose $\a_n$ and $\beta_n$ is an important question. In this paper we will consider the 
method (\ref{SHBM}) with $\a_n$ and $\beta_n$ given by 
\begin{align}\label{ab}
\a_n = \frac{1}{n+2}  \quad \mbox{and} \quad \beta_n = \frac{n}{n+2}
\end{align}
and investigate the convergence behavior of the iterates. The stochastic version of the heavy ball 
method (\ref{HB0}) for the general minimization problems (\ref{MP}), where the objective $f$ is 
a finite sum of convex functions, has been considered in \cite{SGD2021}. The convergence result 
in \cite{SGD2021} however is not applicable to (\ref{SHBM}) when used to solve ill-posed problems 
because of the same reason explained above. Therefore, new analysis should be developed to understand 
the stochastic heavy ball method for ill-posed problems. Like all the other iterative regularization 
methods, when (\ref{SHBM}) is used to solve ill-posed inverse problems, it exhibits the semi-convergence 
phenomenon, i.e. the iterate tends to the sought solution at the beginning, and, after a critical 
number of iterations, the iterate leaves away from the sought solution as the iteration proceeds. 
Thus, properly terminating the iteration is crucial for producing acceptable approximate 
solutions. One may hope to determine a stopping 
index $n_\d$ such that $\|x_{n_\d}^\d - x^\dag\|$ to be as small as possible and 
$\|x_{n_\d}^\d - x^\dag\| \to 0$ as $\d \to 0$. In this paper, by establishing 
a stability estimate and exploring the structure of the method, we propose {\it a priori} 
stopping rules to terminate the method (\ref{SHBM}). When the sought solution $x^\dag$ 
satisfies a benchmark source condition, we derive the estimate on $\EE[\|x_n^\d - 
x^\dag\|^2]$ which enables us to derive a converge rate when the stopping index is properly 
chosen. Furthermore, we extend the method (\ref{SHBM}) to determine solutions of ill-posed 
problems in Banach spaces with special features other than the minimal norm property. 
To the best of our knowledge, our paper is the first study on the stochastic heavy ball 
method as an iterative regularization method in the context of ill-posed inverse problems. 
We hope our work can stimulate interest on this topic which contains many interesting 
and challenging questions that deserve further study. 

This paper is organized as follows. In Section 2 we first establish a stability estimate,
and, when the sought solution satisfies a benchmark source condition, we derive a 
convergence rate result, we then use a density argument to demonstrate the regularization 
property of our method under an {\it a priori} stopping rule without relying on any 
source conditions. In Section \ref{sect3}, we extend the method to cover ill-posed 
inverse problems in Banach spaces with the ability of capturing special feature of sought 
solutions. Finally, in Section \ref{sect4}, we report extensive numerical results to 
illustrate the performance of the method. 

\section{\bf Convergence analysis}\label{sect2}

In this section we will provide a convergence analysis on the method (\ref{SHBM}) with 
$\{\a_n\}$ and $\{\beta_n\}$ given by (\ref{ab}) under {\it a priori} stopping rules. 
For convenience of references, we reformulate the method as the following algorithm.

\begin{algorithm}\label{alg:SHB}
Take an initial guess $x_0 \in X$, choose suitable positive numbers $\eta_i$, $i = 1, \cdots, p$, 
and set $x_0^\d := x_0$. For $n \ge 0$ do the following: 

\begin{enumerate}[leftmargin = 0.8cm]
\item[\emph{(i)}] Pick an index $i_n \in \{1, \cdots, p\}$ randomly via the uniform distribution;

\item[\emph{(ii)}] Update $x_{n+1}^\d$ by (\ref{SHBM}) with $\a_n$ and $\beta_n$ given in (\ref{ab}). 
\end{enumerate}

\end{algorithm}

Note that, once $x_0\in X$ and $\eta_i$, $i=1, \cdots, p$, are fixed, the sequence 
$\{x_n^\d\}$ is completely determined by the sample path $\{i_0, i_1, \cdots\}$; changing 
the sample path can result in a different iterative sequence and thus $\{x_n^\d\}$ is a
random sequence. Let $\F_0 = \emptyset$ and, for each integer $n \ge 1$, let $\F_n$ denote 
the $\sigma$-algebra generated by the random variables $i_k$ for $0\le k < n$. Then 
$\{\F_n: n\ge 0\}$ form a filtration which is natural to Algorithm \ref{alg:SHB}. Let $\EE$ 
denote the expectation associated with this filtration, see \cite{B2020}. The tower property 
$$
\EE[\EE[\phi|\F_n]] = \EE[\phi] \quad \mbox{ for any random variable } \phi
$$
will be frequently used. 

In order to carry out the convergence analysis, we adopt the iterate moving-average viewpoint to
rewrite the method (\ref{SHBM}) and (\ref{ab}) equivalently as 
\begin{align}\label{IMA}
\begin{split}
z_{n+1}^\d & = z_n^\d - \eta_{i_n} A_{i_n}^* (A_{i_n} x_n^\d - y_{i_n}^\d), \\
x_{n+1}^\d & = \frac{n+1}{n+2} x_n^\d + \frac{1}{n+2} z_{n+1}^\d
\end{split}
\end{align}
with $z_0^\d = x_0^\d = x_0$. Furthermore, we also need to consider the sequence 
$\{x_n\}$ defined by Algorithm \ref{alg:SHB} using exact data $y$ in place of the noisy 
data $y^\d$, i.e. 
\begin{align}\label{SHBM-exact}
x_{n+1} = x_n - \a_n \eta_{i_n} A_{i_n}^*(A_{i_n} x_n - y_{i_n}) + \beta_n (x_n - x_{n-1})
\end{align}
with $\a_n$ and $\beta_n$ given by (\ref{ab}) and $x_{-1} = x_0$.  It is easy to see that 
(\ref{SHBM-exact}) can be equivalently stated as   
\begin{align}\label{IMA0}
\begin{split}
z_{n+1} & = z_n - \eta_{i_n} A_{i_n}^* (A_{i_n} x_n - y_{i_n}), \\
x_{n+1} & = \frac{n+1}{n+2} x_n + \frac{1}{n+2} z_{n+1}
\end{split}
\end{align}
with $z_0 = x_0$. In the following analysis we will frequently use the polarization 
identity which states that in any real Hilbert space there holds 
\begin{align}\label{polar}
\|u - v \|^2 = \|u\|^2 + \|v\|^2 - 2\l u, v\r, \quad \forall u, v.     
\end{align}
In particular there holds 
\begin{align}\label{polar2}
2\l u - v, u\r \ge \|u\|^2 - \|v\|^2, \quad \forall u, v. 
\end{align}

\subsection{\bf Stability estimate}

According to the definition of $x_n^\d$ and $x_n$, one can easily see that, for any fixed 
integer $n \ge 0$, there holds $\|x_n^\d - x_n\| \to 0$ as $\d \to 0$ along any sample 
path and thus $\EE[\|x_n^\d - x_n\|^2] \to 0$ as $\d \to 0$. The following result gives 
a quantitative estimate of this kind of stability.  

\begin{lemma}\label{IMA.lem1.1}
Consider the sequences $\{x_n^\d\}$ and $\{x_n\}$ defined by Algorithm \ref{alg:SHB} using noisy 
data and exact data respectively. Assume that 
$0<\eta_i <1/\|A_i\|^2$ for $i = 1, \cdots, p$. Then for all integers $n \ge 0$ there holds 
\begin{align}\label{SHB4}
\EE\left[\|x_n^\d - x_n\|^2\right] \le \frac{\bar \eta}{c_0 p} n \d^2, 
\end{align}
where $\bar \eta:= \max\{\eta_i: i = 1, \cdots, p\}$ and $c_0:= \min\{1-\eta_i \|A_i\|^2: i = 1, \cdots, p\}$. 
\end{lemma}

\begin{proof}
By using the equivalent definition of $\{x_n^\d\}$ and $\{x_n\}$ given by (\ref{IMA}) and (\ref{IMA0}) 
respectively, we first show that 
\begin{align}\label{SHB14}
\EE\left[\|z_n^\d - z_n\|^2\right] \le \frac{\bar \eta}{c_0 p} n \d^2 
\end{align}
for all integers $n \ge 0$. Since $z_0^\d = z_0 = x_0$, it is trivial for $n =0$. We now derive 
(\ref{SHB14}) for all $n \ge 1$. From (\ref{IMA}) and (\ref{IMA0}) it follows that along any 
sample path there hold 
\begin{align}\label{IMA.1}
\begin{split}
z_{n+1}^\d - z_{n+1} & = z_n^\d - z_n - \eta_{i_n} A_{i_n}^* (A_{i_n}(x_n^\d - x_n) - y_{i_n}^\d + y_{i_n}), \\
x_{n+1}^\d - x_{n+1} & = \frac{n+1}{n+2} (x_n^\d - x_n) + \frac{1}{n+2} (z_{n+1}^\d - z_{n+1}).
\end{split}
\end{align}
Therefore, by the polarization identity (\ref{polar}) we have 
\begin{align*}
\|z_{n+1}^\d - z_{n+1}\|^2 
& = \|z_n^\d - z_n - \eta_{i_n} A_{i_n}^* (A_{i_n}(x_n^\d - x_n) - y_{i_n}^\d + y_{i_n})\|^2 \displaybreak[0]\\
& = \|z_n^\d - z_n\|^2 - 2\eta_{i_n} \l z_n^\d - z_n, A_{i_n}^* (A_{i_n} (x_n^\d - x_n) - y_{i_n}^\d + y_{i_n})\r \displaybreak[0]\\
& \quad \, + \eta_{i_n}^2 \|A_{i_n}^*(A_{i_n}(x_n^\d - x_n) - y_{i_n}^\d + y_{i_n})\|^2 \displaybreak[0]\\
& \le \|z_n^\d - z_n\|^2 - 2\eta_{i_n} \l A_{i_n}(z_n^\d - z_n), A_{i_n} (x_n^\d - x_n) - y_{i_n}^\d + y_{i_n}\r \displaybreak[0]\\
& \quad \, + \eta_{i_n}^2 \|A_{i_n}\|^2 \|A_{i_n}(x_n^\d - x_n) - y_{i_n}^\d + y_{i_n}\|^2 \displaybreak[0]\\
& = \|z_n^\d - z_n\|^2 - 2\eta_{i_n} \l A_{i_n}(x_n^\d - x_n), A_{i_n} (x_n^\d - x_n) - y_{i_n}^\d + y_{i_n}\r \displaybreak[0]\\
& \quad \, - 2\eta_{i_n} \l A_{i_n}((z_n^\d - z_n) - (x_n^\d - x_n)), A_{i_n} (x_n^\d - x_n) - y_{i_n}^\d + y_{i_n}\r \displaybreak[0]\\
& \quad \, + \eta_{i_n}^2 \|A_{i_n}\|^2 \|A_{i_n}(x_n^\d - x_n) - y_{i_n}^\d + y_{i_n}\|^2.
\end{align*}
According to the second equation in (\ref{IMA.1}) we have
\begin{align*}
(z_n^\d - z_n) - (x_n^\d - x_n) = n \left((x_n^\d - x_n) - (x_{n-1}^\d - x_{n-1})\right). 
\end{align*}
Thus
\begin{align*}
& \|z_{n+1}^\d  - z_{n+1}\|^2 \\
& \le \|z_n^\d - z_n\|^2 - 2\eta_{i_n} \|A_{i_n} (x_n^\d - x_n) - y_{i_n}^\d + y_{i_n}\|^2 \displaybreak[0]\\
& \quad \, - 2\eta_{i_n} \l y_{i_n}^\d - y_{i_n}, A_{i_n} (x_n^\d - x_n) - y_{i_n}^\d + y_{i_n}\r \displaybreak[0]\\
& \quad \, - 2n\eta_{i_n}\l A_{i_n}((x_n^\d - x_n) - (x_{n-1}^\d - x_{n-1})), A_{i_n} (x_n^\d - x_n)-y_{i_n}^\d + y_{i_n}\r \displaybreak[0]\\
& \quad \, + \eta_{i_n}^2 \|A_{i_n}\|^2 \|A_{i_n}(x_n^\d - x_n) - y_{i_n}^\d + y_{i_n}\|^2.
\end{align*}
By virtue of (\ref{noise}) and the polarization identity (\ref{polar}), we have 
\begin{align*}
& \|z_{n+1}^\d - z_{n+1}\|^2 \\
& \le \|z_n^\d - z_n\|^2 - 2\eta_{i_n} \|A_{i_n} (x_n^\d - x_n) - y_{i_n}^\d + y_{i_n}\|^2 \displaybreak[0]\\
& \quad \, + 2\eta_{i_n} \d_{i_n} \|A_{i_n} (x_n^\d - x_n) - y_{i_n}^\d + y_{i_n}\| \displaybreak[0]\\
& \quad \, - n \eta_{i_n} \left(\|A_{i_n}(x_n^\d - x_n) -y_{i_n}^\d + y_{i_n}\|^2 
- \|A_{i_n} (x_{n-1}^\d - x_{n-1}) -y_{i_n}^\d +y_{i_n}\|^2\right) \displaybreak[0]\\
& \quad \, - n \eta_{i_n} \|A_{i_n}((x_n^\d - x_n) - (x_{n-1}^\d - x_{n-1}))\|^2 \displaybreak[0]\\
& \quad \, + \eta_{i_n}^2 \|A_{i_n}\|^2 \|A_{i_n}(x_n^\d - x_n) - y_{i_n}^\d + y_{i_n}\|^2 \displaybreak[0]\\
& \le \|z_n^\d - z_n\|^2 - (n + 2 - \eta_{i_n} \|A_{i_n}\|^2) \eta_{i_n} \|A_{i_n} (x_n^\d - x_n) - y_{i_n}^\d + y_{i_n}\|^2 \displaybreak[0]\\
& \quad \, + n \eta_{i_n}\|A_{i_n}(x_{n-1}^\d - x_{n-1}) -y_{i_n}^\d +y_{i_n}\|^2 \displaybreak[0]\\
& \quad \, + 2 \eta_{i_n} \d_{i_n} \|A_{i_n} (x_n^\d - x_n) - y_{i_n}^\d + y_{i_n}\|. 
\end{align*}
Taking the conditional expectation on $\F_n$ gives 
\begin{align*}
\EE[\|z_{n+1}^\d - z_{n+1}\|^2|\F_n] 
& \le \|z_n^\d - z_n\|^2  + \frac{n}{p}\sum_{i=1}^p \eta_i \|A_i(x_{n-1}^\d - x_{n-1}) -y_i^\d +y_i\|^2 \displaybreak[0]\\
& \quad \, - \frac{1}{p} \sum_{i=1}^p (n+2-\eta_i \|A_i\|^2) \eta_i \|A_i (x_n^\d - x_n) - y_i^\d + y_i\|^2 \displaybreak[0]\\
& \quad \, + \frac{2}{p} \sum_{i=1}^p \eta_i \d_i \|A_i (x_n^\d - x_n) - y_i^\d + y_i\|. 
\end{align*}
By using the tower property of expectation and noting that $n+2 - \eta_i \|A_i\|^2 \ge n+1+ c_0$
for $i=1, \cdots, p$, we then obtain  
\begin{align*}
\EE[\|z_{n+1}^\d - z_{n+1}\|^2] 
& \le \EE[\|z_n^\d - z_n\|^2]  + \frac{n}{p}\sum_{i=1}^p \eta_i \EE[\|A_i(x_{n-1}^\d - x_{n-1}) -y_i^\d +y_i\|^2] \displaybreak[0]\\
& \quad \, - \frac{n+1 + c_0}{p} \sum_{i=1}^p \eta_i \EE[\|A_i (x_n^\d - x_n) - y_i^\d + y_i\|^2] \displaybreak[0]\\
& \quad \, + \frac{2}{p} \sum_{i=1}^p \eta_i \d_i \EE[\|A_i (x_n^\d - x_n) - y_i^\d + y_i\|]. 
\end{align*}
By the Cauchy-Schwarz inequality we have  
\begin{align*}
\frac{2}{p}\sum_{i=1}^p \eta_i \d_i \EE[\|A_i (x_n^\d - x_n) - y_i^\d + y_i\|]
& \le \frac{2}{p}\sum_{i=1}^p \eta_i \d_i \left(\EE[\|A_i (x_n^\d - x_n) - y_i^\d + y_i\|^2]\right)^{\frac{1}{2}} \displaybreak[0]\\
& \le \frac{c_0}{p} \sum_{i=1}^p \eta_i \EE[\|A_i (x_n^\d - x_n) - y_i^\d + y_i\|^2] \\
& \quad \, + \frac{1}{c_0 p} \sum_{i=1}^p \eta_i \d_i^2.
\end{align*}
Consequently
\begin{align*}
& \EE\left[\|z_{n+1}^\d - z_{n+1}\|^2 + \frac{n+1}{p} \sum_{i=1}^p \eta_i \|A_i (x_n^\d - x_n) - y_i^\d + y_i\|^2\right] \displaybreak[0]\\
& \le \EE\left[\|z_n^\d - z_n\|^2 + \frac{n}{p}\sum_{i=1}^p \eta_i \|A_i(x_{n-1}^\d - x_{n-1}) -y_i^\d +y_i\|^2\right] 
+ \frac{1}{c_0 p} \sum_{i=1}^p \eta_i \d_i^2
\end{align*}
for all integers $n\ge 0$. Since $z_0^\d = z_0$, by recursively using this inequality we thus obtain 
\begin{align*}
\EE\left[\|z_{n+1}^\d - z_{n+1}\|^2 + \frac{n+1}{p} \sum_{i=1}^p \eta_i \|A_i (x_n^\d - x_n) - y_i^\d + y_i\|^2\right] 
& \le  \frac{n+1}{c_0 p} \sum_{i=1}^p \eta_i \d_i^2 \\
& \le \frac{\bar \eta (n+1) \d^2}{c_0 p}
\end{align*}
for all integers $n \ge 0$ which implies (\ref{SHB14}). 

Next we show (\ref{SHB4}). It is trivial for $n=0$ as $x_0^\d = x_0$. 
Assume it is true for some integer $n \ge 0$. It follows from the second equation in 
(\ref{IMA.1}) and the convexity of the function $x \to \|x\|^2$ that 
$$
\|x_{n+1}^\d - x_{n+1}\|^2 \le \frac{n+1}{n+2} \|x_n^\d - x_n\|^2 
+ \frac{1}{n+2}\|z_{n+1}^\d - z_{n+1}\|^2. 
$$
Therefore, by using the induction hypothesis and (\ref{SHB14}), we can obtain 
$\EE[\|x_{n+1}^\d - x_{n+1}\|^2] \le \frac{\bar \eta (n+1) \d^2}{c_0 p}$. 
This completes the proof of (\ref{SHB4}). 
\end{proof}

\subsection{\bf Rate of convergence}

In this subsection we will focus on deriving an estimate on $\EE[\|x_n^\d - x^\dag\|^2]$ 
when the $x_0$-minimal norm solution $x^\dag$ satisfies the benchmark source condition
\begin{align}\label{SC}
x^\dag - x_0 = A^* \la^\dag
\end{align}
for some $\la^\dag := (\la_1^\dag, \cdots, \la_p^\dag) \in Y_1\times \cdots \times Y_p$.
Based on this result we can make an {\it a priori} choice of the stopping index to 
derive a convergence rate in terms of the noise level $\d$. By noting that 
\begin{align*}
\|x_n^\d - x^\dag\|^2 
& \le (\|x_n^\d - x_n\| + \|x_n - x^\dag\|)^2 \le 2 (\|x_n^\d - x_n\|^2 + \|x_n - x^\dag\|^2), 
\end{align*}
we have 
\begin{align}\label{SHB.231}
\EE[\|x_n^\d - x^\dag\|^2] \le 2 \EE[\|x_n^\d - x_n\|^2] + 2 \EE[\|x_n - x^\dag\|^2]. 
\end{align}
Thus, by virtue of Lemma \ref{IMA.lem1.1}, we only need to estimate $\EE[\|x_n - x^\dag\|^2]$
under the source condition (\ref{SC}). To this end, we make use of the equivalent 
formulation of the method (\ref{IMA0}) as follows 
\begin{align}\label{SHB.E10}
x_n = x_0 + A^* \la_n,
\end{align}
where $\{\la_n\}$ is a sequence in $Y_1 \times \cdots \times Y_p$ defined as follows: 
$\la_{-1} = \la_0 =0$ and when $\la_n =(\la_{n,1}, \cdots, \la_{n, p})$ is defined, we calculate 
$x_n$ by (\ref{SHB.E10}) and define 
$\la_{n+1} :=(\la_{n+1,1}, \cdots, \la_{n+1, p})$ by 
\begin{align}\label{SHB10}
\la_{n+1, i} = \left\{\begin{array}{lll}
\la_{n, i} + \beta_n (\la_{n, i} - \la_{n-1, i}) & \mbox{if } i \ne i_n,\\[1ex]
\la_{n, i_n} + \beta_n (\la_{n, i_n} - \la_{n-1, i_n})
- \a_n \eta_{i_n}(A_{i_n} x_n - y_{i_n}) & \mbox{if } i = i_n. 
\end{array}\right.
\end{align}
One can easily check that this defined sequence $\{x_n\}$ coincides with the one defined by 
(\ref{IMA0}). We have the following result.

\begin{lemma}\label{SHBM:lem1}
Consider the sequence $\{x_n\}$ defined by Algorithm \ref{alg:SHB} using exact data and let 
$\{\la_n\}$ be defined as in (\ref{SHB10}) with $\la_{-1} = \la_0 =0$. Assume that 
the source condition (\ref{SC}) holds. Let 
$$
u_n := \la_n + n(\la_n - \la_{n-1}) \quad \mbox{and} \quad 
r_n := \sum_{i=1}^n \frac{1}{\eta_i} \|u_{n, i} - \la_i^\dag\|^2
$$
for all integers $n \ge 0$, where $u_{n,i}$ denotes the $i$-th component 
of $u_n$ in $Y_i$. Then
\begin{align*}
\EE\left[r_{n+1} + \frac{n+2}{p} \|x_n - x^\dag\|^2\right] 
& \le \EE\left[r_n + \frac{n}{p} \|x_{n-1} - x^\dag\|^2\right] \\
& \quad \, + \frac{1}{p} \sum_{i=1}^p \eta_i \EE\left[\|A_i x_n - y_i\|^2\right].
\end{align*}
\end{lemma}

\begin{proof}
By the definition of $u_{n+1}$ and the formulae of $\a_n$ and $\beta_n$, it is easy 
to see that 
\begin{align*}
u_{n+1, i} = \left\{\begin{array}{lll}
u_{n, i} & \mbox{ if } i \ne i_n,\\
u_{n, i_n} - \eta_{i_n} (A_{i_n} x_n - y_{i_n}) & \mbox{ if } i = i_n. 
\end{array}\right.
\end{align*}
Thus, by the polarization identity (\ref{polar}) and $A_{i_n} x^\dag = y_{i_n}$ we have  
\begin{align*}
r_{n+1} & = \sum_{i\ne i_n} \frac{1}{\eta_i} \|u_{n+1, i} - \la_i^\dag\|^2
+ \frac{1}{\eta_{i_n}} \|u_{n+1, i_n} - \la_{i_n}^\dag\|^2 \displaybreak[0]\\
& = \sum_{i\ne i_n} \frac{1}{\eta_i} \|u_{n, i} - \la_i^\dag\|^2 
+ \frac{1}{\eta_{i_n}} \|u_{n, i_n} - \la_{i_n}^\dag - \eta_{i_n} (A_{i_n} x_n - y_{i_n})\|^2 \displaybreak[0]\\
& = \sum_{i\ne i_n} \frac{1}{\eta_i} \|u_{n, i} - \la_i^\dag\|^2 
+ \frac{1}{\eta_{i_n}} \|u_{n, i_n} - \la_{i_n}^\dag\|^2 + \eta_{i_n} \|A_{i_n} x_n - y_{i_n}\|^2 \displaybreak[0]\\
& \quad \, - 2 \l u_{n, i_n} - \la_{i_n}^\dag, A_{i_n} (x_n - x^\dag) \r \displaybreak[0]\\
& = r_n + \eta_{i_n} \|A_{i_n} x_n - y_{i_n}\|^2 
- 2 \l A_{i_n}^*(u_{n, i_n} - \la_{i_n}^\dag), x_n - x^\dag\r.
\end{align*}
Taking the conditional expectation on $\F_n$ gives 
\begin{align*}
\EE[r_{n+1}|\F_n]  
& = r_n + \frac{1}{p} \sum_{i=1}^p \eta_i \|A_i x_n - y_i\|^2 
- \frac{2}{p} \EE\left[\left\l \sum_{i=1}^p A_i^*(u_{n, i} - \la_{i}^\dag), x_n - x^\dag\right\r\right] \displaybreak[0]\\
& = r_n + \frac{1}{p} \sum_{i=1}^p \eta_i \|A_i x_n - y_i\|^2 
- \frac{2}{p} \left\l A^*(u_n - \la^\dag), x_n - x^\dag\right\r.
\end{align*}
Consequently    
\begin{align*}
\EE[r_{n+1}] = \EE[\EE[r_{n+1}|\F_n]]  
& = \EE[r_n] + \frac{1}{p} \sum_{i=1}^p \eta_i \EE[\|A_i x_n - y_i\|^2] \\
& \quad \, - \frac{2}{p} \EE\left[\left\l A^*(u_n - \la^\dag), x_n - x^\dag\right\r\right].
\end{align*}
By the definition of $u_n$, (\ref{SC}) and (\ref{SHB.E10}) we have 
$A^* (u_n - \la^\dag) = x_n - x^\dag + n(x_n - x_{n-1})$. Therefore 
\begin{align*}
\EE[r_{n+1}] & = \EE[r_n] + \frac{1}{p} \sum_{i=1}^p \eta_i \EE[\|A_i x_n - y_i\|^2] 
- \frac{2}{p} \EE\left[\|x_n - x^\dag\|^2\right] \displaybreak[0]\\
& \quad \, - \frac{2n}{p} \EE\left[\left\l x_n - x_{n-1}, x_n - x^\dag\right\r\right] \displaybreak[0]\\
& \le \EE[r_n] + \frac{1}{p} \sum_{i=1}^p \eta_i \EE[\|A_i x_n - y_i\|^2] 
- \frac{2}{p} \EE\left[\|x_n - x^\dag\|^2\right] \displaybreak[0]\\
& \quad \, - \frac{n}{p} \EE\left[\|x_n - x^\dag\|^2 - \|x_{n-1} - x^\dag\|^2\right],
\end{align*}
where for the last step we used (\ref{polar2}). Regrouping the terms thus completes the proof.  
\end{proof}

In order to use Lemma \ref{SHBM:lem1} to derive an estimate on $\EE[\|x_n - x^\dag\|^2]$
with $\eta_i$ being as large as possible, we need the following result which enables us to cancel 
out the summation term on the right hand side of the inequality in Lemma \ref{SHBM:lem1}. 

\begin{lemma}\label{SHBM:lem2}
Let $\{x_n\}$ and $\{z_n\}$ be defined by (\ref{IMA0}). Assume $0< \eta_i <1/\|A_i\|^2$ 
for $i = 1, \cdots, p$. Then 
\begin{align*}
& \EE\left[\|z_{n+1} - x^\dag\|^2 + \frac{n+1}{p} \sum_{i=1}^p \eta_i \|A_i x_n - y_i\|^2\right] \displaybreak[0]\\
& \le \EE\left[\|z_n -x^\dag\|^2 + \frac{n}{p} \sum_{i=1}^p \eta_i \|A_i x_{n-1} - y_i\|^2\right]
- \frac{c_0}{p} \sum_{i=1}^p \eta_i \EE[\|A_i x_n - y_i\|^2],
\end{align*}
where $c_0$ is the positive constant defined in Lemma \ref{IMA.lem1.1}.
\end{lemma}

\begin{proof}
By the definition of $z_{n+1}$ and the polarization identity (\ref{polar}) we have 
\begin{align*}
\|z_{n+1} - x^\dag\|^2
& = \|z_n - x^\dag - \eta_{i_n} A_{i_n}^*(A_{i_n} x_n - y_{i_n})\|^2 \displaybreak[0]\\
&= \|z_n - x^\dag\|^2 - 2\eta_{i_n} \l A_{i_n}^* (A_{i_n} x_n - y_{i_n}), z_n - x^\dag\r \displaybreak[0]\\
&\quad \, + \eta_{i_n}^2 \| A_{i_n}^* (A_{i_n} x_n - y_{i_n})\|^2.
\end{align*}
Note that $z_n = x_n + n (x_n - x_{n-1})$ and $A_{i_n} x^\dag = y_{i_n}$. Therefore, 
we may use (\ref{polar2}) to obtain 
\begin{align*}
& 2 \eta_{i_n} \l A_{i_n}^* (A_{i_n} x_n - y_{i_n}), z_n - x^\dag\r \displaybreak[0]\\
& = 2 \eta_{i_n} \l A_{i_n} x_n - y_{i_n}, A_{i_n} x_n - y_{i_n}\r  
+ 2 n \eta_{i_n} \l A_{i_n} x_n - y_{i_n}, A_{i_n}(x_n - x_{n-1})\r \displaybreak[0]\\
& \ge 2\eta_{i_n} \|A_{i_n} x_n - y_{i_n}\|^2 + n \eta_{i_n} \left(\|A_{i_n} x_n - y_{i_n}\|^2 
- \|A_{i_n} x_{n-1} - y_{i_n}\|^2\right).
\end{align*}
Consequently 
\begin{align*}
\|z_{n+1} - x^\dag\|^2
&\leq \|z_n - x^\dag\|^2 + n\eta_{i_n} \|A_{i_n} x_{n-1} - y_{i_n}\|^2 \displaybreak[0] \nonumber \\
&\quad \, - \left(n+2 - \eta_{i_n} \|A_{i_n}\|^2\right) \eta_{i_n} \|A_{i_n} x_n - y_{i_n}\|^2.
\end{align*}
Taking the full expectation gives  
\begin{align}\label{IMA.10}
\EE[\|z_{n+1} - x^\dag\|^2] 
& = \EE[\EE[\|z_{n+1} - x^\dag\|^2|\F_n ]] \nonumber \\ 
& \le \EE[\|z_n - x^\dag\|^2] + \frac{n}{p} \sum_{i=1}^p \eta_i \EE[\|A_i x_{n-1} - y_i\|^2] \displaybreak[0]\nonumber \\
& \quad \, - \frac{1}{p} \sum_{i=1}^p (n+2 - \eta_i \|A_i\|^2) \eta_i 
\EE[\|A_i x_n - y_i\|^2] \displaybreak[0]\nonumber \\
& \le \EE[\|z_n - x^\dag\|^2] + \frac{n}{p} \sum_{i=1}^p \eta_i \EE[\|A_i x_{n-1} - y_i\|^2] \displaybreak[0]\nonumber \\
& \quad \, - \frac{n+1+ c_0}{p} \sum_{i=1}^p \eta_i \EE[\|A_i x_n - y_i\|^2].
\end{align}
Regrouping the terms we thus obtain the desired inequality. 
\end{proof}

By using Lemma \ref{SHBM:lem1} and Lemma \ref{SHBM:lem2}, we can obtain the following 
estimate on $\EE[\|x_n - x^\dag\|^2]$ under the source condition (\ref{SC}). 

\begin{lemma}\label{SHBM.thm1}
Consider the sequence $\{x_n\}$ defined by Algorithm \ref{alg:SHB} using exact data.
Assume $0<\eta_i <1/\|A_i\|^2$ for $i = 1, \cdots, p$. If the $x_0$-minimal norm 
solution $x^\dag$ satisfies the source condition (\ref{SC}), then 
\begin{align*}
\EE\left[\|x_n - x^\dag\|^2\right] \le \frac{p M_0}{c_0(n+1)} 
\end{align*}
for all integers $n \ge 0$, where $c_0>0$ is the constant from Lemma \ref{IMA.lem1.1} and 
$M_0:= \|x_0-x^\dag\|^2 + c_0 \sum_{i=1}^p \frac{1}{\eta_i} \|\la_i^\dag\|^2$. 
\end{lemma}

\begin{proof}
By virtue of Lemma \ref{SHBM:lem1} and Lemma \ref{SHBM:lem2} we have 
\begin{align*}
& \EE\left[\|z_{n+1} - x^\dag\|^2 + \frac{n+1}{p} \sum_{i=1}^p \eta_i \|A_i x_n - y_i\|^2 
+ c_0 \left(r_{n+1} + \frac{n+1}{p} \|x_n - x^\dag\|^2\right)\right] \displaybreak[0]\\
& \le \EE\left[\|z_n - x^\dag\|^2 + \frac{n}{p} \sum_{i=1}^p \eta_i \|A_i x_{n-1} - y_i\|^2 
+ c_0 \left(r_n + \frac{n}{p} \|x_{n-1} - x^\dag\|^2\right)\right] 
\end{align*}
for all integers $n =0, 1, \cdots$. Recursively using this inequality gives 
\begin{align*}
& \EE\left[\|z_{n+1} - x^\dag\|^2 + \frac{n+1}{p} \sum_{i=1}^p \eta_i \|A_i x_n - y_i\|^2 
+ c_0 \left(r_{n+1} + \frac{n+1}{p} \|x_n - x^\dag\|^2\right)\right] \\
& \le  \EE\left[\|z_0 - x^\dag\|^2 + c_0 r_0 \right] 
= \|x_0-x^\dag\|^2 + c_0 \sum_{i=1}^p \frac{1}{\eta_i} \|\la_i^\dag\|^2 =:M_0. 
\end{align*}
This in particular implies that 
\begin{align*}
\frac{(n+1) c_0}{p} \EE[\|x_n - x^\dag\|^2] \le M_0
\end{align*}
which shows the desired estimate.
\end{proof}

Based on Lemma \ref{IMA.lem1.1} and Lemma \ref{SHBM.thm1}, we can now prove the following 
convergence rate result on the stochastic gradient method with heavy-ball momentum
under an {\it a priori} stopping rule. 

\begin{theorem}\label{SHB.thm}
Consider the sequence $\{x_n^\d\}$ defined by Algorithm \ref{alg:SHB}. 
Assume $0<\eta_i <1/\|A_i\|^2$ for $i = 1, \cdots, p$. If the $x_0$-minimal norm 
solution $x^\dag$ satisfies the source condition (\ref{SC}), then 
\begin{align*}
\EE\left[\|x_n^\d - x^\dag\|^2\right] \le \frac{2 p M_0}{c_0(n+1)} + \frac{2\bar \eta}{c_0 p} n \d^2 
\end{align*}
for all integers $n \ge 0$. Consequently, if the integer $n_\d$ is chosen such that 
$(n_\d + 1)/p \sim \d^{-1}$, then 
\begin{align*}
\EE\left[\|x_{n_\d}^\d - x^\dag\|^2\right] \le C \d, 
\end{align*}
where $C$ is a constant depending only on $c_0$, $\bar \eta$ and $M_0$. 
\end{theorem}

\begin{proof}
By invoking the inequality (\ref{SHB.231}), we can use Lemma \ref{IMA.lem1.1} and 
Lemma \ref{SHBM.thm1} to conclude the proof immediately. 
\end{proof}

\subsection{\bf Convergence}

In Theorem \ref{SHB.thm} we have established a convergence rate result for Algorithm \ref{alg:SHB}  
when the $x_0$-minimal norm solution $x^\dag$ satisfies the source condition (\ref{SC}). This source 
condition might be too strong to be satisfied in applications. It is necessary to establish 
a convergence result on Algorithm \ref{alg:SHB} without using any source condition on $x^\dag$. 
In this subsection we will prove the following convergence result. 

\begin{theorem}\label{SHB:thm4}
Consider the sequence $\{x_n^\d\}$ defined by Algorithm \ref{alg:SHB}.  
Assume that $0<\eta_i <1/\|A_i\|^2$ for $i = 1, \cdots, p$ and let $x^\dag$ denote the unique 
$x_0$-minimal norm solution of (\ref{sys}). Then for the integer $n_\d$ 
chosen such that $n_\d \to \infty$ and $\d^2 n_\d \to 0$ as $\d \to 0$ there holds 
$$
\EE[\|x_{n_\d}^\d - x^\dag\|^2] \to 0 \quad \mbox{ as } \d \to 0.
$$
\end{theorem}

Considering the stability estimate given in Lemma \ref{IMA.lem1.1}, we will establish Theorem \ref{SHB:thm4} 
by showing that $\EE[\|x_n - x^\dag\|^2] \to 0$ as $n \to \infty$ for the sequence $\{x_n\}$ defined by 
Algorithm \ref{alg:SHB} using exact data. Due to the randomness of $i_n$ and the momentum term 
$\beta_n (x_n -x_{n-1})$ involved in (\ref{SHBM-exact}), 
it is highly nontrivial to derive $\EE[\|x_n - x^\dag\|^2] \to 0$ straightforwardly. We will 
use a perturbation argument developed in \cite{Jin2010,Jin2011}. Namely, as an $x_0$-minimal 
norm solution, there holds $x^\dag - x_0 \in \mbox{Null}(A)^\perp = \overline{\mbox{Ran}(A^*)}$, 
and thus we may choose $\hat x_0 \in X$ as close to $x_0$ as we want such that 
$x^\dag - \hat x_0 \in \mbox{Ran}(A^*)$. Define $\{\hat x_n\}$ by 
\begin{align}\label{SHB3}
\hat x_{n+1} = \hat x_n - \a_n A_{i_n}^*(A_{i_n} \hat x_n - y_{i_n}) + \beta_n (\hat x_n - \hat x_{n-1}) 
\end{align}
with $\hat x_{-1} = \hat x_0$ and the index $i_n \in \{1, \cdots, p\}$ chosen uniformly at random. 
We will establish $\EE[\|x_n - x^\dag\|^2] \to 0$ as $n\to \infty$ 
by deriving estimates on $\EE[\|\hat x_n - x^\dag\|^2]$ and $\EE[\|x_n - \hat x_n\|^2]$. 

For $\EE[\|\hat x_n - x^\dag\|^2]$ we can apply the same argument in the proof of 
Lemma \ref{SHBM.thm1} to the sequence $\{\hat x_n\}$ to obtain the following result.

\begin{lemma}\label{IMA.lem1.2}
Consider the sequence $\{\hat x_n\}$ defined by (\ref{SHB3}), where $\hat x_0$ is chosen such that 
$x^\dag - \hat x_0 \in \emph{Ran}(A^*)$. Assume that $0<\eta_i <1/\|A_i\|^2$ for $i = 1, \cdots, p$. 
Then for any integer $n\ge 0$ there holds 
\begin{align*}
\EE[\|\hat x_n - x^\dag\|^2] \le \frac{p \hat M_0}{c_0 (n+1)},
\end{align*}
where $\hat M_0:= \|\hat x_0-x^\dag\|^2 + c_0 \sum_{i=1}^p \frac{1}{\eta_i} \|\hat \la_i^\dag\|^2$ 
and $\hat \la^\dag :=(\hat \la_1^\dag, \cdots, \hat \la_p^\dag) \in Y_1\times \cdots \times Y_p$ 
is such that $x^\dag - \hat x_0 = A^* \hat \la^\dag$. 
\end{lemma}

We next derive estimate on $\EE[\|x_n - \hat x_n\|^2]$ in terms of $\|x_0 - \hat x_0\|^2$.  
We have the following result. 

\begin{lemma}\label{IMA.lem1.3}
Consider the sequences $\{x_n\}$ and $\{\hat x_n\}$ defined by (\ref{SHBM-exact}) and (\ref{SHB3}) 
respectively. Assume that $0<\eta_i \le 1/\|A_i\|^2$ for $i = 1, \cdots, p$. Then for all $n \ge 0$ 
there holds 
\begin{align}\label{SHB.71}
\EE[\|x_n - \hat x_n\|^2] \le \|x_0 - \hat x_0\|^2. 
\end{align}
\end{lemma}

\begin{proof}
Note that the sequence $\{\hat x_n\}$ from (\ref{SHB3}) can be equivalently defined by 
\begin{align}\label{IMA3}
\begin{split}
\hat z_{n+1} & = \hat z_n - \eta_{i_n} A_{i_n}^* (A_{i_n} \hat x_n - y_{i_n}), \\
\hat x_{n+1} & = \frac{n+1}{n+2} \hat x_n + \frac{1}{n+2} \hat z_{n+1}
\end{split}
\end{align}
with $\hat z_0 = \hat x_0$. We first use (\ref{IMA0}) and (\ref{IMA3}) to show that 
\begin{align}\label{IMA11}
\EE[\|z_{n} - \hat z_{n}\|^2] \le \|x_0 - \hat x_0\|^2, \quad \forall n \ge 0
\end{align}
which is trivial for $n=0$. In order to establish (\ref{IMA11}) for all integers $n\ge 1$, by
virtue of (\ref{IMA0}) and (\ref{IMA3}) we can see that along any sample path there holds  
\begin{align}\label{IMA03}
\begin{split}
z_{n+1} - \hat z_{n+1} & = z_n - \hat z_n - \eta_{i_n} A_{i_n}^* A_{i_n} (x_n - \hat x_n), \\
x_{n+1} - \hat x_{n+1} & = \frac{n+1}{n+2} (x_n - \hat x_n) + \frac{1}{n+2} (z_{n+1} - \hat z_{n+1}).
\end{split}
\end{align}
Therefore, by using the polarization identity (\ref{polar}) we can obtain 
\begin{align*}
\|z_{n+1} - \hat z_{n+1}\|^2 
& = \|z_n - \hat z_n - \eta_{i_n} A_{i_n}^* A_{i_n}(x_n - \hat x_n)\|^2 \displaybreak[0]\\
& = \|z_n - \hat z_n\|^2 + \eta_{i_n}^2 \|A_{i_n}^* A_{i_n}(x_n - \hat x_n)\|^2 \displaybreak[0]\\  
& \quad \, - 2\eta_{i_n} \l z_n - \hat z_n, A_{i_n}^* A_{i_n} (x_n - \hat x_n)\r \displaybreak[0]\\
& = \|z_n - \hat z_n\|^2 + \eta_{i_n}^2 \|A_{i_n}^*A_{i_n}(x_n - \hat x_n)\|^2  \displaybreak[0]\\
& \quad \, -2\eta_{i_n} \l (z_n - \hat z_n) - (x_n - \hat x_n), A_{i_n}^*A_{i_n} (x_n - \hat x_n)\r \displaybreak[0]\\
& \quad \, -2\eta_{i_n} \l x_n - \hat x_n, A_{i_n}^* A_{i_n}(x_n - \hat x_n)\r.
\end{align*}
Using the second equation in (\ref{IMA03}) and the polarization identity (\ref{polar}) we further have 
\begin{align*}
\|z_{n+1} - \hat z_{n+1}\|^2 
& \le \|z_n - \hat z_n\|^2 + \eta_{i_n}^2 \|A_{i_n}\|^2 \|A_{i_n} (x_n - \hat x_n)\|^2  \displaybreak[0]\\
& \quad \, -2n\eta_{i_n} \l A_{i_n}(x_n - \hat x_n) - A_{i_n}(x_{n-1} - \hat x_{n-1}), A_{i_n}(x_n - \hat x_n)\r \displaybreak[0]\\
& \quad \, -2\eta_{i_n} \|A_{i_n} (x_n - \hat x_n)\|^2 \displaybreak[0]\\
& = \|z_n - \hat z_n\|^2 + \eta_{i_n}^2 \|A_{i_n}\|^2 \|A_{i_n}(x_n - \hat x_n)\|^2  \displaybreak[0]\\
& \quad \, -n \eta_{i_n} \left(\|A_{i_n}(x_n - \hat x_n)\|^2 - \|A_{i_n} (x_{n-1} - \hat x_{n-1})\|^2\right) \displaybreak[0]\\
& \quad \, -n \eta_{i_n} \|A_{i_n}(x_n-\hat x_n) - A_{i_n}(x_{n-1} - \hat x_{n-1})\|^2 \displaybreak[0]\\
& \quad \, -2\eta_{i_n} \|A_{i_n}(x_n - \hat x_n)\|^2  \displaybreak[0]\\
& \le \|z_n - \hat z_n\|^2 + n \eta_{i_n} \|A_{i_n} (x_{n-1} - \hat x_{n-1})\|^2 \displaybreak[0]\\
& \quad \, - (n + 2 - \eta_{i_n}\|A_{i_n}\|^2) \eta_{i_n} \|A_{i_n} (x_n - \hat x_n)\|^2.
\end{align*}
Taking the full expectation gives 
\begin{align*}
\EE[\|z_{n+1} - \hat z_{n+1}\|^2] 
& = \EE[\EE[\|z_{n+1} - \hat z_{n+1}\|^2 | \F_n]] \displaybreak[0] \\
& \le \EE[\|z_n - \hat z_n\|^2] + \frac{n}{p} \sum_{i=1}^p \eta_i \EE[\|A_i (x_{n-1} - \hat x_{n-1})\|^2] \displaybreak[0]\\
& \quad \, - \frac{1}{p} \sum_{i=1}^p (n + 2 - \eta_i\|A_i\|^2) \eta_i \EE[\|A_i (x_n - \hat x_n)\|^2].
\end{align*}
Since $0< \eta_i \le 1/\|A_i\|^2$ for $i = 1, \cdots, p$, we can obtain
\begin{align*}
& \EE[\|z_{n+1} - \hat z_{n+1}\|^2] + \frac{n+1}{p} \sum_{i=1}^p \eta_i \EE[\|A_i (x_n - \hat x_n)\|^2] \displaybreak[0]\\
& \le \EE[\|z_n - \hat z_n\|^2] + \frac{n}{p} \sum_{i=1}^p \eta_i \EE[\|A_i (x_{n-1} - \hat x_{n-1})\|^2]
\end{align*}
for all integers $n\ge 0$. Recursively using this inequality gives
\begin{align*}
\EE[\|z_{n+1} - \hat z_{n+1}\|^2] + \frac{n+1}{p} \sum_{i=1}^p \eta_i \EE[\|A_i (x_n - \hat x_n)\|^2] 
\le \EE[\|z_0 - \hat z_0\|^2] = \|x_0 - \hat x_0\|^2
\end{align*}
which shows (\ref{IMA11}). 

Next we show (\ref{SHB.71}). It is trivial for $n = 0$. Assume 
it is also true for some $n$. By noting from the second equation in (\ref{IMA03}) that 
$x_{n+1} - \hat x_{n+1}$ is a convex combination of $x_n - \hat x_n$ and 
$z_{n+1} - \hat z_{n+1}$, we can obtain 
$$
\|x_{n+1} - \hat x_{n+1}\|^2 \le \frac{n+1}{n+2} \|x_n - \hat x_n\|^2
+ \frac{1}{n+2} \|z_{n+1} - \hat z_{n+1}\|^2
$$
Thus, we may use the induction hypothesis and (\ref{IMA11}) to obtain 
\begin{align*}
& \EE[\|x_{n+1} - \hat x_{n+1}\|^2] \\
& \le \frac{n+1}{n+2} \EE[\|x_n - \hat x_n\|^2]
+ \frac{1}{n+2} \EE[\|z_{n+1} - \hat z_{n+1}\|^2] \le \|x_0 - \hat x_0\|^2. 
\end{align*}
The proof is therefore complete. 
\end{proof}

Based on Lemma \ref{IMA.lem1.2} and Lemma \ref{IMA.lem1.3}, we can now prove the convergence of 
Algorithm \ref{alg:SHB} using exact data. 

\begin{theorem}\label{SHBM.thm2}
Consider the sequence $\{x_n\}$ defined by Algorithm \ref{alg:SHB} using exact data. Assume that 
$0<\eta_i <1/\|A_i\|^2$ for $i = 1, \cdots, p$ and let $x^\dag$ denote the unique $x_0$-minimal 
norm solution of (\ref{sys}). Then 
$$
\lim_{n\to \infty} \EE\left[\|x_n - x^\dag\|^2\right] = 0.
$$
\end{theorem}

\begin{proof}
Since $x^\dag - x_0 \in \overline{\mbox{Ran}(A^*)}$, for any $\ep>0$ we 
can find $\hat x_0 \in X$ such that $\|x_0-\hat x_0\| < \ep$ and 
$x^\dag - \hat x_0 \in \mbox{Ran}(A^*)$. Define $\{\hat x_n\}$ as in (\ref{SHB3}). Then we have from Lemma \ref{IMA.lem1.3} that 
$$
\EE\left[\|x_n - \hat x_n\|^2\right] \le \|x_0 - \hat x_0\|^2 < \ep^2. 
$$
Moreover, it follows from Lemma \ref{IMA.lem1.2} that 
$$
\EE\left[\|\hat x_n - x^\dag\|^2\right] \le C (n+1)^{-1}
$$
for some constant $C$ which may depend on $\ep$ but is independent of $n$. Consequently
\begin{align*}
\EE[\|x_n - x^\dag\|^2] 
& \le \EE\left[(\|x_n - \hat x_n\| + \|\hat x_n - x^\dag\|)^2\right] \\
& \le 2\EE\left[\|x_n - \hat x_n\|^2 + \|\hat x_n - x^\dag\|^2\right] \\
& \le 2 \ep^2 + 2C (n+1)^{-1}.  
\end{align*}
Therefore 
$$
\limsup_{n\to \infty} \EE\left[\|x_n - x^\dag\|^2\right] \le 2\ep^2. 
$$
Since $\ep>0$ is arbitrary, we must have $\EE\left[\|x_n - x^\dag\|^2\right] \to 0$ 
as $n \to \infty$. 
\end{proof}

By using Theorem \ref{SHBM.thm2} and Lemma \ref{IMA.lem1.1} we are now ready to prove Theorem \ref{SHB:thm4}, the main convergence result on Algorithm \ref{alg:SHB}. 

\begin{proof}[Proof of Theorem \ref{SHB:thm4}]
According to (\ref{SHB.231}) we have 
\begin{align*}
\EE\left[\|x_{n_\d}^\d - x^\dag\|^2\right] 
& \le 2 \EE\left[\|x_{n_\d}^\d - x_{n_\d}\|^2\right] 
+ 2 \EE\left[\|x_{n_\d} - x^\dag\|^2\right].
\end{align*}
Since $n_\d\to \infty$, we may use Theorem \ref{SHBM.thm2} to obtain 
$$
\EE\left[\|x_{n_\d} - x^\dag\|^2\right] \to 0 \quad \mbox{ as } \d \to 0. 
$$
By using Lemma \ref{IMA.lem1.1} and $\d^2 n_\d \to 0$, we also have 
$$
\EE\left[\|x_{n_\d}^\d - x_{n_\d}\|^2\right]
\le \frac{\bar \eta}{c_0 p} \d^2 n_\d \to 0 \quad \mbox{ as }\d\to 0. 
$$
Therefore $\EE[\|x_{n_\d}^\d - x^\dag\|^2] \to 0$ as $\d \to 0$. 
\end{proof}

Due to the ill-posedness of the underlying problems, any iterative regularization method
using noisy data, include Algorithm \ref{alg:SHB}, exhibits the semi-convergence property.
Therefore, a proper termination of the iteration is crucial for a satisfactory reconstruction. 
In Theorem \ref{SHBM.thm1} and Theorem \ref{SHBM.thm2} we provided convergence analysis of 
Algorithm \ref{alg:SHB} under {\it a priori} choices of the stopping index, where $\eta_{i_n}$ 
is chosen by
\begin{equation}\label{noDP}
    \eta_{i_n} = \mu_0/\|A_{i_n}\|^2, \quad 0<\mu_0<1;
\end{equation}
In particular, Theorem \ref{SHBM.thm1} indicates that the choice of the stopping 
index depends crucially on the information of the source condition which is usually unknown. 
A wrong guess of the source condition may lead to a bad choice of the stopping index and hence 
a bad reconstruction result. For deterministic iterative regularization methods, {\it a posteriori} 
stopping rules have been widely used to resolve this issue and the discrepancy principle is the 
most prominent one (\cite{EHN1996}). In \cite{JLZ2023} the discrepancy principle has been 
successfully incorporated into the stochastic mirror descent method, including the stochastic 
gradient descent method as a special case, to remove the semi-convergence property. It is natural 
to expect that this can also been done for Algorithm \ref{alg:SHB}. Assuming the noise levels 
$\d_i$, $i = 1, \cdots, p$, are known, we may incorporate the discrepancy principle into 
Algorithm \ref{alg:SHB} by taking $\eta_{i_n}$ as follows:
\begin{align}\label{DP}
\eta_{i_n} = \left\{\begin{array}{lll}
\mu_0/\|A_{i_n}\|^2 & \mbox{ if } \|A_{i_n} x_n^\d - y_{i_n}^\d\|>\tau \d_{i_n}, \\
0 & \mbox{ otherwise},
\end{array} \right.
\end{align}
where $0<\mu_0<1$ and $\tau>1$ are user specified constants. Although no convergence theory is 
available for Algorithm \ref{alg:SHB} with $\eta_{i_n}$ chosen by (\ref{DP}), we will provide 
numerical simulations to demonstrate that this procedure indeed can reduce the effect of semi-convergence 
property.

\section{\bf Extension}\label{sect3}

In the previous section we have considered the stochastic heavy ball method for solving 
linear ill-posed system (\ref{sys}) in Hilbert spaces and demonstrated the convergence of the 
iterates to the unique $x_0$-minimal norm solution. In many applications, the sought solution 
may not sit in a Hilbert space and it may have some special feature to be recovered. 
Algorithm \ref{alg:SHB} does not have the capability to deal with these issues. In this section 
we will extend Algorithm \ref{alg:SHB} to cover these situations. Consider again the linear system 
(\ref{sys}) but with each $A_i: X \to Y_i$ being a bounded linear operator from a fixed Banach 
spaces $X$ to a Hilbert space $Y_i$. We introduce a proper, lower semi-continuous, 
strongly convex function $\R: X \to (-\infty, \infty]$ with the aim of capturing the feature of the 
sought solution. Here $\R$ is called $\mu$-strongly convex for some constant $\mu>0$ if  
$$
\R(t \bar x + (1-t) x) + \mu t (1-t)\|\bar x - x\|^2 \le t \R(\bar x) + (1-t) \R(x)
$$
for all $\bar x, x \in \mbox{dom}(\R)$ and $0\le t\le 1$, where $\mbox{dom}(\R):=\{x\in X: \R(x) < \infty\}$
is called the effective domain of $\R$. We then consider finding a solution $x^\dag$ of (\ref{sys}) 
such that 
\begin{align}\label{SHB.B1}
\R(x^\dag) = \min\left\{\R(x): A_i x = y_i, \ i = 1, \cdots, p\right\}
\end{align}
which is called a $\R$-minimizing solution of (\ref{sys}). Assuming (\ref{sys}) has a solution in 
$\mbox{dom}(\R)$, by a standard argument it is easy to see that the $\R$-minimizing solution 
exists and is unique. We will determine the unique $\R$-minimizing solution $x^\dag$ approximately 
by extending the stochastic heavy ball method. To this end, we consider the dual problem 
of (\ref{SHB.B1}) which is given by 
\begin{align}\label{dual}
\min_{\la \in Y} \left\{\R^*(A^* \la) - \l \la, y\r\right\},
\end{align}
where $\R^*: X^* \to (-\infty, \infty]$ denotes the convex conjugate of $\R$, i.e. 
$$
\R^*(\xi) := \sup_{x\in X} \left\{\l \xi, x\r - \R(x)\right\}, \quad \forall \xi \in X^*. 
$$
Since $\R$ is strongly convex, $\R^*$ is continuous differentiable, its gradient 
$\nabla \R^*$ maps $X^*$ to $X$ and 
$$
\|\nabla \R^*(\bar \xi) - \nabla \R^*(\xi)\| \le \frac{\|\bar \xi- \xi\|}{2 \mu}, \quad 
\forall \bar \xi, \xi\in X^*; 
$$
see \cite[Corollary 3.5.11]{Z2002}. Applying the heavy ball method (\ref{HB0}) to solve 
(\ref{dual}) gives 
\begin{align*}
\la_{n+1} = \la_n - \a_n \eta (A \nabla \R^*(A^* \la_n) - y) + \beta_n (\la_n - \la_{n-1}),
\end{align*}
where $\eta>0$ is a suitable constant. 
By setting $\xi_n := A^* \la_n$ and $x_n = \nabla \R^*(A^* \la_n)$, we then obtain 
\begin{align*}
x_n = \nabla \R^*(\xi_n), \quad 
\xi_{n+1} = \xi_n - \a_n \eta A^* (A x_n - y) + \beta_n (\xi_n - \xi_{n-1}).
\end{align*}
In case only noisy data $y^\d$ is available, we may replace $y$ by $y^\d$ in the above equation 
to obtain the method
\begin{align*}
x_n^\d = \nabla \R^*(\xi_n^\d), \quad 
\xi_{n+1}^\d = \xi_n^\d - \a_n \eta A^* (A x_n^\d - y^\d) + \beta_n (\xi_n^\d - \xi_{n-1}^\d).
\end{align*}
Note that 
\begin{align}\label{SHB.B2}
A^*(A x_n^\d - y^\d) = \sum_{i=1}^p A_i^* (A_i x_n^\d - y_i^\d).
\end{align}
Therefore, the implementation of the above method requires to calculate $A_i^*(A_i x_n^\d - y_i^\d)$
for all $i = 1, \cdots, p$ at each iteration step; in case $p$ is huge, this requires a huge amount of 
computational work. In order to reduce the computational load per iteration, we may consider using 
a randomly selected term from the right hand side of (\ref{SHB.B2}) and then carry out the iteration step. 
Repeating this procedure leads to the following stochastic method for solving (\ref{SHB.B1}) in 
Banach spaces. 

\begin{algorithm}\label{alg:SHB2}
Let $\xi_{-1}^\d = \xi_0^\d = 0 \in X^*$ and take suitable numbers $\eta_i$ for $i = 1, \cdots, p$. 
For $n \ge 0$ do the following:

\begin{enumerate}[leftmargin = 0.8cm]
\item[\emph{(i)}] Calculate $x_n^\d$ by 
$$
x_n^\d = \nabla \R^*(\xi_n^\d) = \arg\min\{\R(x) - \l \xi_n^\d, x\r\}; 
$$

\item[\emph{(ii)}] Pick an index $i_n\in \{1, \cdots, p\}$ randomly via the uniform distribution;

\item[\emph{(iii)}] Update $\xi_{n+1}^\d$ by 
\begin{align*}
\xi_{n+1}^\d & = \xi_n^\d - \a_n \eta_{i_n} A_{i_n}^*(A_{i_n} x_n^\d - y_{i_n}^\d) 
+ \beta_n (\xi_n^\d - \xi_{n-1}^\d)
\end{align*}
with $\a_n$ and $\beta_n$ given by (\ref{ab}). 
\end{enumerate}
\end{algorithm}

According to the definition of $x_n^\d$ and the subdifferential calculus, we have 
$\xi_n^\d \in \p \R(x_n^\d)$, where $\p \R$ denotes the subdifferential of $\R$, i.e. 
$$
\p\R(x) := \{\xi \in X^*: \R(\bar x) \ge \R(x) + \l \xi, \bar x -x\r \mbox{ for all } \bar x \in X\}, 
\quad \forall x \in X.
$$
In order to analyze the method, we use the Bregman distance induced by $\R$. Given $x\in X$ 
with $\p \R(x) \ne \emptyset$ and $\xi \in \p \R(x)$ we define 
$$
D_\R^\xi(\bar x, x) := \R(\bar x) - \R(x) - \l \xi, \bar x - x\r, \quad \forall \bar x \in X
$$
which is called the Bregman distance induced by $\R$ at $x$ in the direction $\xi$. 

In the following we derive an estimate on $\EE[\|x_n^\d - x^\dag\|^2]$ under the 
benchmark source condition 
\begin{align}\label{SC2}
\xi^\dag:= A^* \la^\dag \in \p \R(x^\dag)
\end{align}
on the sought solution $x^\dag$, where 
$\la^\dag := (\la_1^\dag, \cdots, \la_p^\dag) \in Y:= Y_1\times \cdots \times Y_p$.
We need the following elementary result (see \cite{JLZ2023}). 

\begin{lemma}\label{rbdgm.lem4}
Let $\{a_n\}$ and $\{b_n\}$ be two sequences of nonnegative numbers such that
$$
a_n^2 \le b_n^2 + c \sum_{j=0}^{n-1} a_j, \quad n=0, 1,\cdots,
$$
where $c \ge 0$ is a constant. If $\{b_n\}$ is non-decreasing, then
$$
a_n \le b_n + c n, \quad n=0, 1, \cdots.
$$
\end{lemma}

\begin{theorem}
Consider Algorithm \ref{alg:SHB2} and assume that $0<\eta_i < 2\mu/\|A\|^2$ for $i = 1, \cdots, p$. 
If the unique $\R$-minimizing solution $x^\dag$ satisfies the source condition (\ref{SC2}), 
then there exists a positive constant $C$ depending only on $\mu$, $\|A\|^2$ and $\eta_i$ such that  
$$
\EE[\|x_n^\d - x^\dag\|^2] \le C \left(\frac{pM_0}{n+1} + \frac{1}{p}(n+1) \d^2 + \d^2\right),
$$
where $M_0:= \sum_{i=1}^p \frac{1}{\eta_i} \|\la_i^\dag\|^2$. Consequently, if the integer $n_\d$ is 
chosen such that $(n_\d + 1)/p \sim \d^{-1}$, then 
$$
\EE[\|x_{n_\d}^\d - x^\dag\|^2] = O(\d). 
$$
\end{theorem}

\begin{proof}
It is easy to see that $\{x_n^\d\}$ and $\{\xi_n^\d\}$ from Algorithm \ref{alg:SHB2} 
can be equivalently written as 
$$
x_n^\d = \nabla \R^*(\xi_n^\d), \qquad \xi_{n+1}^\d = A^* \la_{n+1}^\d,
$$
where $\la_{n+1}^\d := (\la_{n+1, 1}^\d, \cdots, \la_{n+1,p}^\d) \in Y$ is defined by 
\begin{align*}
\la_{n+1, i}^\d = \left\{\begin{array}{lll}
\la_{n, i}^\d + \beta_n (\la_{n, i}^\d - \la_{n-1, i}^\d) & \mbox{if } i \ne i_n,\\[1ex]
\la_{n, i_n}^\d + \beta_n (\la_{n, i_n}^\d - \la_{n-1, i_n})^\d
- \a_n \eta_{i_n}(A_{i_n} x_n^\d - y_{i_n}^\d) & \mbox{if } i = i_n
\end{array}\right.
\end{align*}
with $\la_{-1} = \la_0 =0$. As in Lemma \ref{SHBM:lem1} we define  
$$
u_n^\d := \la_n^\d + n(\la_n^\d - \la_{n-1}^\d) \quad \mbox{and} \quad 
r_n^\d := \sum_{i=1}^n \frac{1}{\eta_i} \|u_{n, i}^\d - \la_i^\dag\|^2
$$
for all integers $n \ge 0$, where $u_{n,i}^\d$ denotes the $i$-th component of $u_n^\d$ in $Y_i$.
Following the proof of Lemma \ref{SHBM:lem1} we can obtain 
\begin{align*}
\EE[r_{n+1}^\d]  
& \le \EE[r_n^\d] + \frac{1}{p} \sum_{i=1}^p \eta_i \EE[\|A_i x_n^\d - y_i^\d\|^2] 
- \frac{2}{p} \EE\left[\left\l A^*(u_n^\d - \la^\dag), x_n^\d - x^\dag\right\r\right] \\
& \quad \, + \frac{2\bar \eta^{1/2}}{p} \d  
\left(\EE[r_n^\d]\right)^{1/2}.
\end{align*}
By the definition of $u_n^\d$ and $\xi_n^\d$, we have $A^* (u_n^\d - \la^\dag) = \xi_n^\d - \xi^\dag 
+ n(\xi_n^\d - \xi_{n-1}^\d)$. Consequently 
\begin{align*}
\EE[r_{n+1}^\d]  
& \le \EE[r_n^\d] + \frac{1}{p} \sum_{i=1}^p \eta_i \EE[\|A_i x_n^\d - y_i^\d\|^2] 
- \frac{2}{p} \EE\left[\left\l \xi_n^\d - \xi^\dag, x_n^\d - x^\dag\right\r\right] \\
& \quad \, - \frac{2n}{p} \EE\left[\l \xi_n^\d - \xi_{n-1}^\d, x_n^\d - x^\dag\r\right] 
+ \frac{2\bar \eta^{1/2}}{p} \d  \left(\EE[r_n^\d]\right)^{1/2}.
\end{align*}
By straightforward calculation we can see that 
$$
\l \xi_n^\d - \xi^\dag, x_n^\d - x^\dag\r = D_\R^{\xi_n^\d} (x^\dag, x_n^\d) 
+ D_\R^{\xi^\dag}(x_n^\d, x^\dag) 
$$
and
\begin{align*}
\l \xi_n^\d - \xi_{n-1}^\d, x_n^\d - x^\dag\r 
& = D_\R^{\xi_n^\d} (x^\dag, x_n^\d) - D_\R^{\xi_{n-1}^\d} (x^\dag, x_{n-1}^\d) 
+ D_\R^{\xi_{n-1}^\d} (x_n^\d, x_{n-1}^\d)\\
& \ge D_\R^{\xi_n^\d} (x^\dag, x_n^\d) - D_\R^{\xi_{n-1}^\d} (x^\dag, x_{n-1}^\d).
\end{align*}
Therefore 
\begin{align*}
\EE[r_{n+1}^\d]  
& \le \EE[r_n^\d] + \frac{1}{p} \sum_{i=1}^p \eta_i \EE[\|A_i x_n^\d - y_i^\d\|^2] 
- \frac{2}{p} \EE\left[D_\R^{\xi_n^\d} (x^\dag, x_n^\d) 
+ D_\R^{\xi^\dag}(x_n^\d, x^\dag) \right] \\
& \quad \, - \frac{2n}{p} \EE\left[D_\R^{\xi_n^\d} (x^\dag, x_n^\d) 
- D_\R^{\xi_{n-1}^\d} (x^\dag, x_{n-1}^\d)\right] 
+ \frac{2\bar \eta^{1/2}}{p} \d  \left(\EE[r_n^\d]\right)^{1/2}.
\end{align*}
Letting $\Delta_n^\d := D_\R^{\xi_n^\d}(x^\dag, x_n^\d)$ and regrouping the terms, we can obtain  
\begin{align}\label{SHB.B4}
\EE\left[r_{n+1}^\d + \frac{2(n+1)}{p} \Delta_n^\d \right] 
& \le \EE\left[r_n^\d + \frac{2n}{p} \Delta_{n-1}^\d \right] 
+ \frac{1}{p} \EE\left[\sum_{i=1}^p \eta_i \|A_i x_n^\d - y_i^\d\|^2\right] \nonumber \\
& \quad \, - \frac{2}{p} D_\R^{\xi^\dag}(x_n^\d, x^\dag) 
+ \frac{2\bar \eta^{1/2}}{p} \d  \left(\EE[r_n^\d]\right)^{1/2}.
\end{align}
Note that 
\begin{align*}
\sum_{i=1}^p \eta_i \|A_i x_n^\d - y_i^\d\|^2 
& \le \bar\eta \sum_{i=1}^p \left(\|A_i (x_n^\d - x^\dag)\| + \d_i\right)^2 \\
& \le \bar\eta \sum_{i=1}^p \left((1+\ep) \|A_i (x_n^\d - x^\dag)\|^2 + (1+ \ep^{-1}) \d_i^2\right) \\
& = (1+\ep) \bar\eta \|A(x_n^\d - x^\dag)\|^2 + (1+\ep^{-1}) \bar\eta \d^2 \\
& \le (1+\ep) \bar \eta \|A\|^2 \|x_n^\d - x^\dag\|^2 + (1+\ep^{-1}) \bar\eta \d^2,
\end{align*}
where $\ep>0$ is any number. Since $\R$ is $\mu$-strongly convex, we have 
$$
D_\R^{\xi^\dag}(x_n^\d, x^\dag) \ge \mu \|x_n^\d - x^\dag\|^2.
$$
Combining the above two equations with (\ref{SHB.B4}) gives 
\begin{align*}
\EE\left[r_{n+1}^\d + \frac{2(n+1)}{p} \Delta_n^\d \right] 
& \le \EE\left[r_n^\d + \frac{2n}{p} \Delta_{n-1}^\d \right] 
- \frac{1}{p} \left(2\mu - (1+\ep) \bar\eta \|A\|^2\right) \|x_n^\d - x^\dag\|^2 \\
& \quad \, + \frac{1}{p} (1+\ep^{-1}) \bar\eta \d^2 
+ \frac{2\bar \eta^{1/2}}{p} \d  \left(\EE[r_n^\d]\right)^{1/2}.
\end{align*}
According to the given condition on $\eta_i$ we have $\bar \eta < 2\mu/\|A\|^2$. 
We may take $\ep := 2\mu/(\bar\eta\|A\|^2)-1>0$ to obtain 
\begin{align*}
\EE\left[r_{n+1}^\d + \frac{2(n+1)}{p} \Delta_n^\d \right] 
\le \EE\left[r_n^\d + \frac{2n}{p} \Delta_{n-1}^\d \right] + \frac{1}{p} C_0 \d^2 
+ \frac{2\bar \eta^{1/2}}{p} \d  \left(\EE[r_n^\d]\right)^{1/2}, 
\end{align*}
where $C_0:= (1+\ep^{-1}) \bar\eta$. Recursively using this inequality gives 
\begin{align}\label{SHBM.21}
\EE\left[r_{n+1}^\d + \frac{2(n+1)}{p} \Delta_n^\d\right] 
\le M_0 + \frac{(n+1)C_0\d^2}{p} 
+ \frac{2 \bar \eta^{1/2}}{p} \d \sum_{k=0}^n \left(\EE[r_k^\d]\right)^{1/2},
\end{align}
where 
\begin{align*}
M_0 := \EE\left[r_0^\d \right] 
= \sum_{i=1}^p \frac{1}{\eta_i} \|\la_i^\dag\|^2. 
\end{align*}
This in particular implies that 
\begin{align*}
\EE[r_{n+1}^\d] \le M_0 + \frac{(n+1)C_0\d^2}{p}
+ \frac{2\bar \eta^{1/2}}{p} \d \sum_{k=0}^n (\EE[r_k^\d])^{1/2}.
\end{align*}
Thus, by using Lemma \ref{rbdgm.lem4} we can obtain 
\begin{align*}
(\EE[r_k^\d])^{1/2} 
\le \left(M_0 + \frac{C_0 k}{p} \d^2\right)^{1/2} + \frac{2\bar \eta^{1/2}}{p} k \d 
\le M_0^{1/2} + \left(\frac{C_0 k}{p}\right)^{1/2} \d  
+ \frac{2\bar \eta^{1/2}}{p} k \d
\end{align*}
for all integers $k \ge 0$. Consequently 
\begin{align*}
\sum_{k=0}^n (\EE[r_k^\d])^{1/2} 
& \le M_0^{1/2} (n+1) + \left(\frac{C_0}{p}\right)^{1/2} \d \sum_{k=0}^n k^{1/2}   
+ \frac{2\bar \eta^{1/2}}{p} \d \sum_{k=0}^n k \displaybreak[0]\\
& \le M_0^{1/2} (n+1) 
+ \left(\frac{2C_0^{1/2} (n+1)^{3/2}}{3 p^{1/2}} 
+ \frac{\bar\eta^{1/2} (n+1)^2}{p} \right) \d.
\end{align*}
Combining this with (\ref{SHBM.21}) we have 
\begin{align*}
\frac{2(n+1)}{p} \EE[\Delta_n^\d] 
& \le M_0 + \frac{(n+1)C_0 \d^2}{p} + \frac{2 (\bar \eta M_0)^{1/2}}{p} (n+1) \d \displaybreak[0]\\
& \quad \, + 2 \left(\frac{2C_0^{1/2} (n+1)^{3/2}}{3 \bar\eta^{1/2} p^{3/2}} 
+ \frac{(n+1)^2}{p^2} \right) \bar \eta \d^2 \displaybreak[0]\\
& \le 2 M_0 + \frac{(n+1)C_0 \d^2}{p} + \frac{\bar\eta}{p^2} (n+1)^2 \d^2 \displaybreak[0]\\
& \quad \, + 2 \left(\frac{2C_0^{1/2} (n+1)^{3/2}}{3\bar\eta^{1/2} p^{3/2}} 
+ \frac{(n+1)^2}{p^2} \right) \bar \eta \d^2. 
\end{align*}
Therefore, by the strong convexity of $\R$, we have 
\begin{align*}
\mu \EE[\|x_n^\d - x^\dag\|^2] \le \EE[\Delta_n^\d] 
& \le \frac{p M_0}{2(n+1)} + \frac{1}{2} C_0 \d^2 + \frac{\bar\eta}{2p} (n+1) \d^2 \displaybreak[0]\\
& \quad \, + \left(\frac{2C_0^{1/2} (n+1)^{1/2}}{3\bar\eta^{1/2}p^{1/2}}
+ \frac{n+1}{p} \right) \bar \eta \d^2. 
\end{align*}
By using $2C_0^{1/2}(n+1)^{1/2}/(3\bar\eta^{1/2} p^{1/2}) \le C_0/(9 \bar\eta) + (n+1)/p$ 
we further obtain  
\begin{align*}
\mu \EE[\|x_n^\d - x^\dag\|^2] 
& \le \frac{p M_0}{2(n+1)} + \frac{5 \bar\eta}{2p} (n+1) \d^2  + \frac{11}{18} C_0\d^2
\end{align*}
which completes the proof.
\end{proof}

\begin{remark}
Without using any source condition on $x^\dag$, it is expected for Algorithm \ref{alg:SHB2}
that if $n_\d$ is chosen such that $n_\d \to \infty$ and $\d^2 n_\d \to 0$ then 
$\EE[\|x_{n_\d}^\d - x^\dag\|^2] \to 0$ as $\d \to 0$. However, we do not have a proof about 
this statement.
\end{remark}

\begin{remark}
As usual, Algorithm \ref{alg:SHB2} with $\eta_{i_n} = \mu_0/\|A\|^2$ for some $0<\mu_0<2\mu$
demonstrates the semi-convergence property. To suppress such effect, we may incorporate the 
discrepancy principle into the choice of $\eta_{i_n}$, namely, when the noise levels $\d_i$, $i=1, \cdots, p$, are available, we may use $\eta_{i_n}$ chosen by (\ref{DP}) in Algorithm \ref{alg:SHB2}. 
Although there is no theoretical justification, numerical simulations indicate that such 
a choice of $\eta_{i_n}$ indeed can reduce the appearance of semi-convergence significantly.
\end{remark}

\section{\bf Numerical simulations}\label{sect4}

In this section we provide numerical results to demonstrate the performance of the stochastic 
gradient method with heavy-ball momentum. We consider the first kind integral equation
$$
\int_a^b \kappa(s, t) x(t) dt = y(s), \quad s\in [a,b]
$$
with discrete data, where $\kappa \in C([a,b]\times [a,b])$. The data $y_i = y(s_i)$, 
$i = 1, \cdots, p$, are collected at the $p$ sample points $s_1, \cdots, s_p$ in $[a,b]$ with
$s_i = a + (i-1)(b-a)/(p-1)$. Correspondingly  we have the linear system
\begin{align}\label{IES}
A_i x:=\int_a^b \kappa (s_i, t) x(t) dt = y_i,  \quad i = 1, \cdots, p. 
\end{align}

\begin{example}
We illustrate the performance of Algorithm \ref{alg:SHB} by applying it to solve the 
linear system (\ref{IES}) in which $[a,b] = [-6, 6]$ and $\kappa(s,t) = \varphi(s-t)$ with 
$\varphi(s) = \left(1+\cos(\pi s/3)\right) \chi_{\{|s|<3\}}$. Assume the sought solution is 
\begin{align*}
x^\dag(t) = \sin(\pi t/12) + \sin(\pi t/3) + \frac{1}{200} t^2 (1-t).
\end{align*}
Instead of the exact data $y = (y_1, \cdots, y_p)$ with $y_i := A_i x^\dag$, we use the noisy data  
\begin{align}\label{nd}
y_i^\d = y_i + \d_{rel} \|y\|_\infty \ep_i, \quad i = 1, \cdots, p, 
\end{align}
where $\d_{rel}$ is the relative noise level and $\ep_i$ are random noises obeying 
the uniform distribution on $[-1,1]$. We execute Algorithm \ref{alg:SHB} by using the 
initial guess $x_{-1}^\d = x_0^\d= 0$ and setting $\eta_i= 0.6/\left\|A_i\right\|^2$ for 
$i=1,\cdots, p$, which follows \eqref{noDP} with $\mu_0 = 0.6$. We also use noisy data 
with three distinct relative noise levels $\d_{rel} = 10^{-1}, 10^{-2}$ and $10^{-3}$. 
In our implementation, all integrals over $[-6,6]$ are approximated by the trapezoidal 
rule based on the $m=1000$ nodes partitioning $[-6,6]$ into subintervals of equal length.
In Figure \ref{fig:SHB1} we plot the corresponding reconstruction errors versus the 
number of iterations; the first row plots the relative mean square errors 
$\EE[\|x_n^\d-x^\dag\|_{L^2}^2/\|x^\dag\|_{L^2}^2]$ which are calculated 
approximately by the average of $100$ independent runs and the second row plots 
$\|x_n^\d - x^\dag\|_{L^2}^2/\|x^\dag\|_{L^2}^2$ for a particular individual run. 
From these plots we can easily observe that the method demonstrates the semi-convergence 
property no matter how small the relative noise level is. Furthermore, the 
semi-convergence occurs earlier and the iterates diverge faster if using a noisy data 
with larger relative noise level. Therefore, in order to produce a good approximate 
solution, it is crucial to terminate the iteration properly. Theorem \ref{SHB:thm4} 
provides a general {\it a priori} criterion for terminating the method and Theorem \ref{SHB.thm}
provides an {\it a priori} stopping rule in case the source condition (\ref{SC}) is known to hold. 

\begin{figure}[htpb]
\centering
\includegraphics[width = 0.32\textwidth]{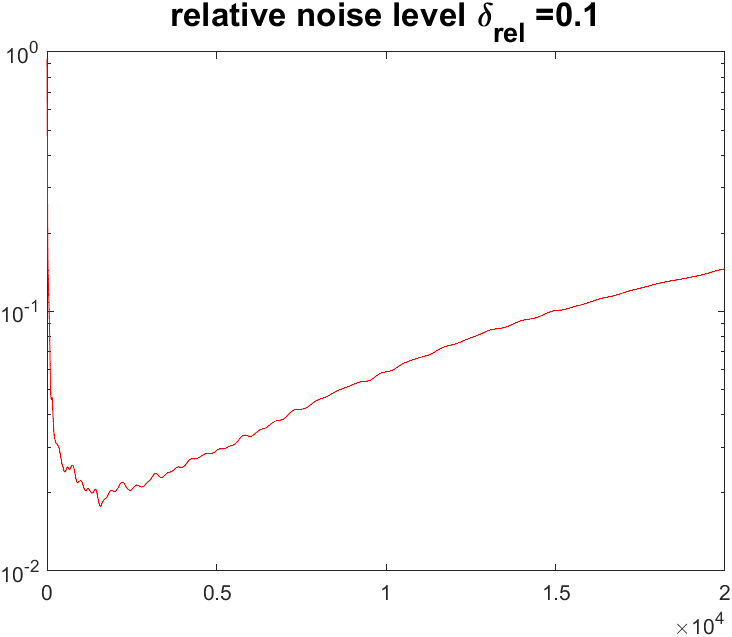}
\includegraphics[width = 0.32\textwidth]{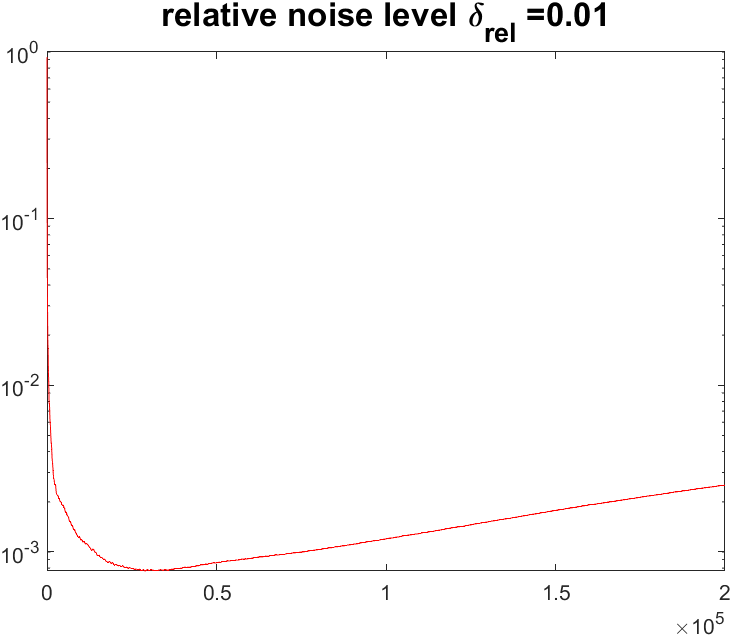}
\includegraphics[width = 0.32\textwidth]{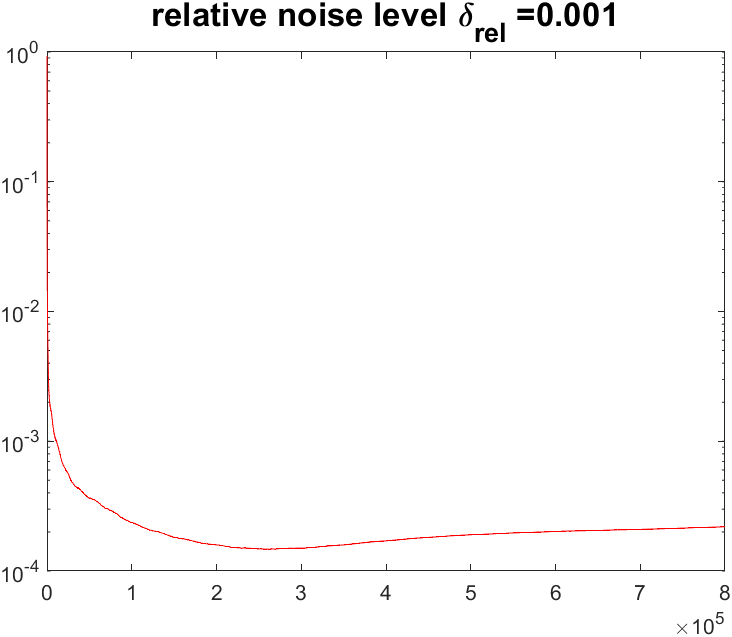}
\includegraphics[width = 0.32\textwidth]{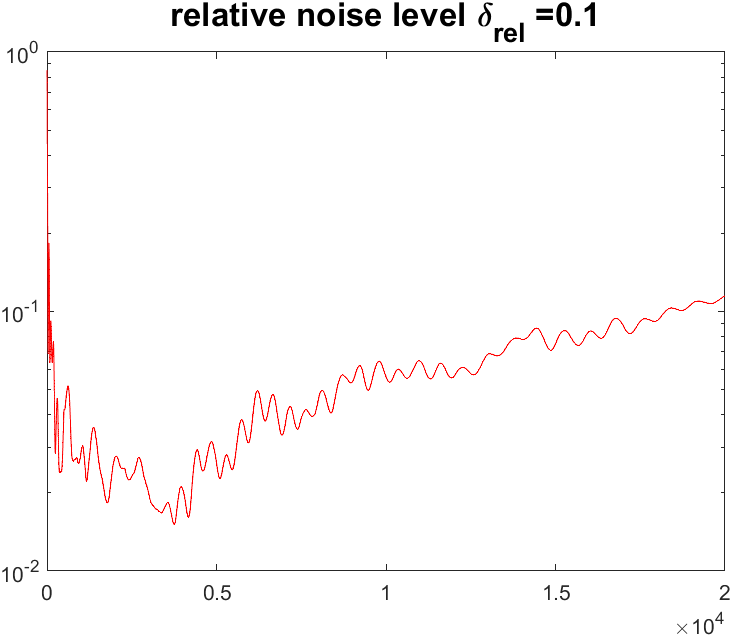}
\includegraphics[width = 0.32\textwidth]{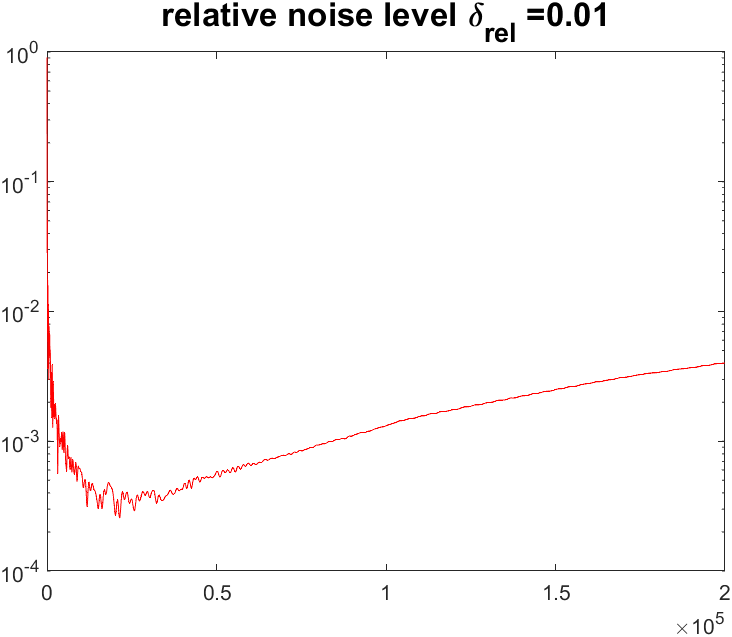}
\includegraphics[width = 0.32\textwidth]{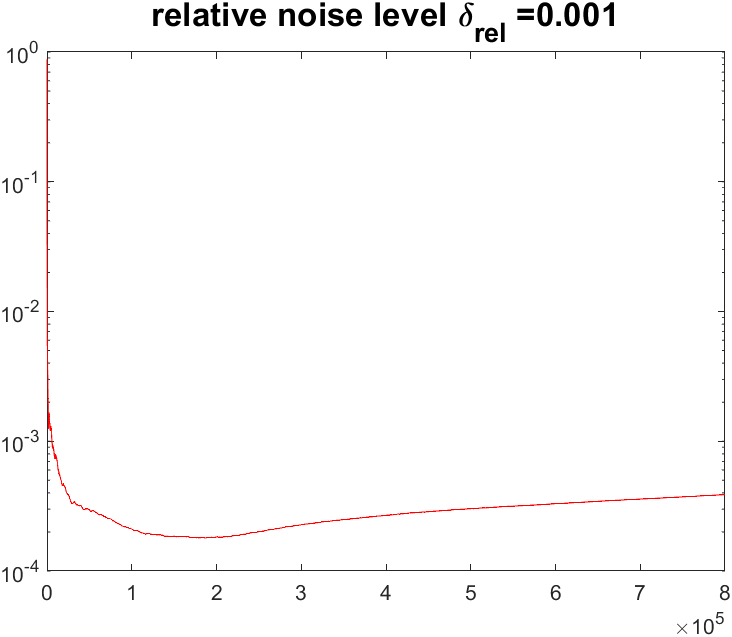}
\caption{Illustration of the semi-convergence of Algorithm \ref{alg:SHB} using noisy data with various relative noise levels} \label{fig:SHB1}
\end{figure}

In order to reduce or even remove the semi-convergence phenomenon, we next consider incorporating 
the discrepancy principle into the choice of $\eta_{i_n}$. Assume the noise level 
$\d_i := \d_{rel} \|y\|_\infty$, $i = 1, \cdots, p$, are known, we use the step size 
given by \eqref{DP}. To illustrate the advantage of incorporating the discrepancy principle, 
we execute Algorithm \ref{alg:SHB} under the two different choices of $\eta_{i_n}$ given 
by \eqref{noDP} and \eqref{DP} with $\mu_0=0.6$ and $\tau=1.4$. In Figure \ref{fig:SHB2} we 
plot the reconstruction errors generated by Algorithm \ref{alg:SHB} using noisy data for three 
different relative noise levels $\delta_{rel} = 10^{-1}, 10^{-2}$ and $10^{-3}$, where 
``\texttt{SHB}'' and ``\texttt{SHB-DP}'' stand for the results corresponding to the 
$\eta_{i_n}$ chosen by \eqref{noDP} and \eqref{DP} respectively. The averaged relative 
error $\|x_n^{\delta} - x^\dag\|_{L^2}^2/\|x^\dag\|_{L^2}^2$ for $100$ independent runs, 
which are used as approximations of $\EE[\|x_n^{\delta} - x^\dag\|_{L^2}^2/\|x^\dag\|_{L^2}^2]$,
are plotted in the first row and the relative error for a particular individual run is 
presented in the second row. These results indicate that using $\eta_{i_n}$ chosen by \eqref{DP} 
can efficiently reduce the effect of semi-convergence which is in particular striking when 
noisy data with large noise level are used.

\begin{figure}[htpb]
\centering
\includegraphics[width = 0.32\textwidth]{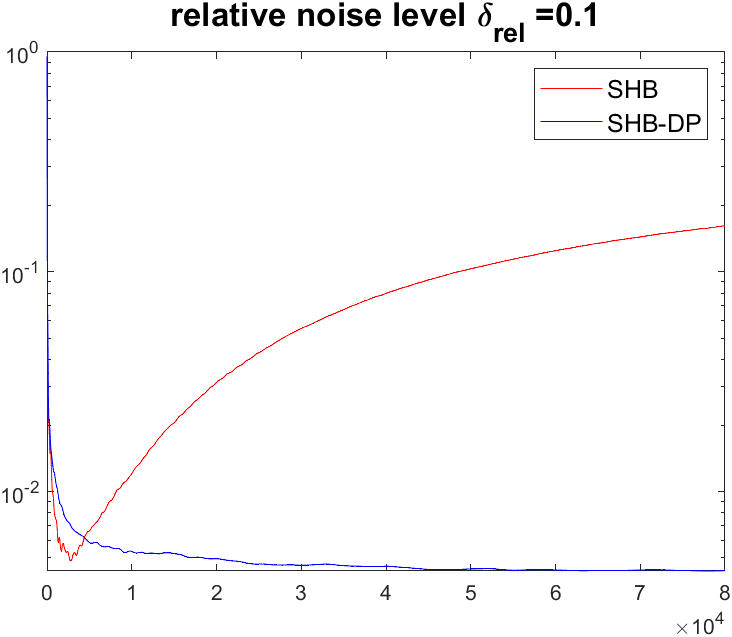}
\includegraphics[width = 0.32\textwidth]{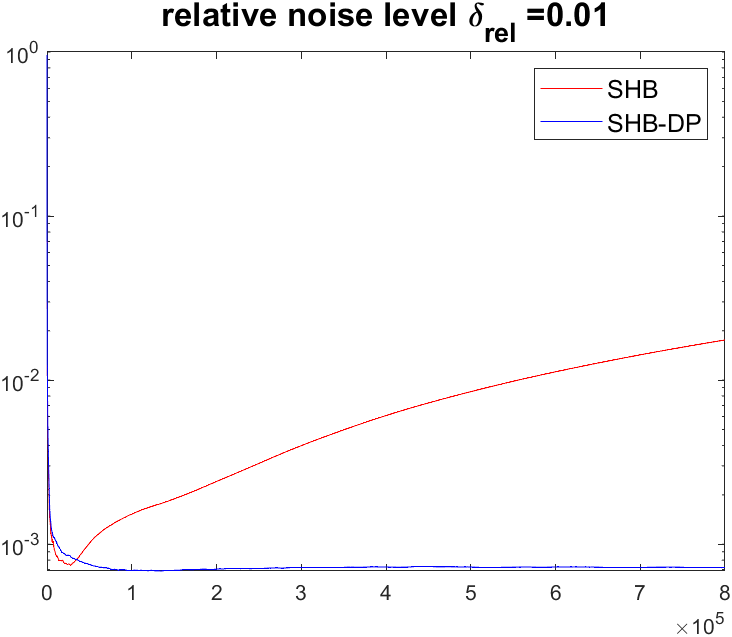}
\includegraphics[width = 0.32\textwidth]{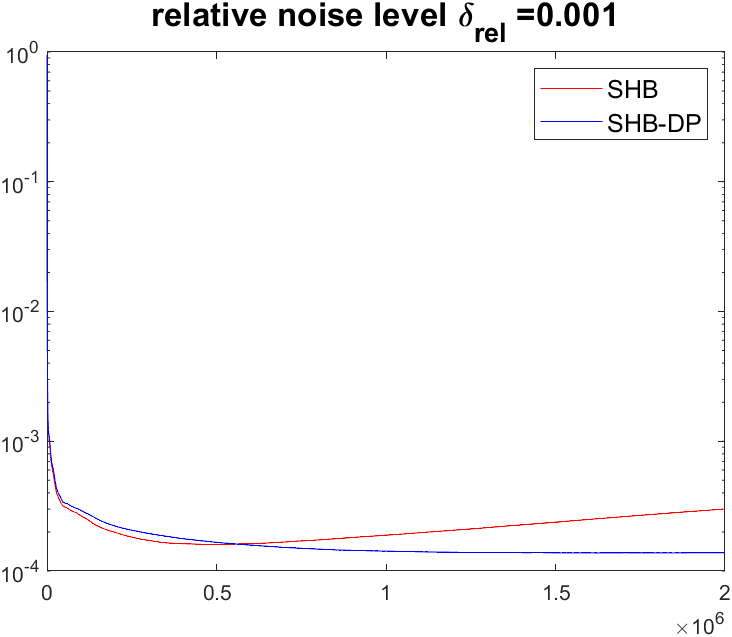}
\includegraphics[width = 0.32\textwidth]{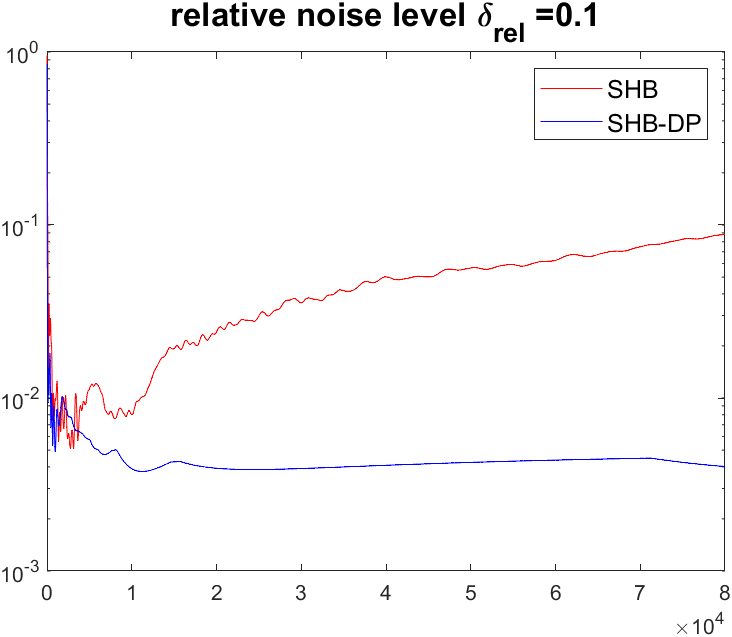}
\includegraphics[width = 0.32\textwidth]{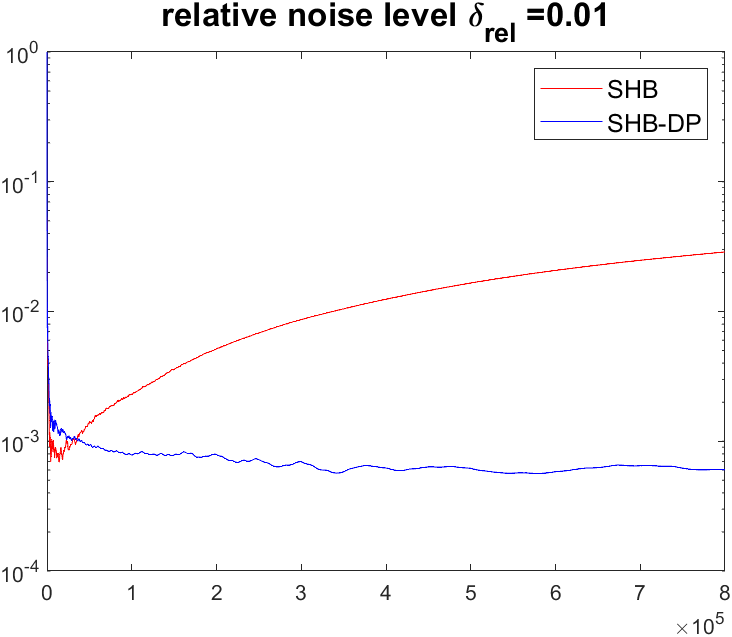}
\includegraphics[width = 0.32\textwidth]{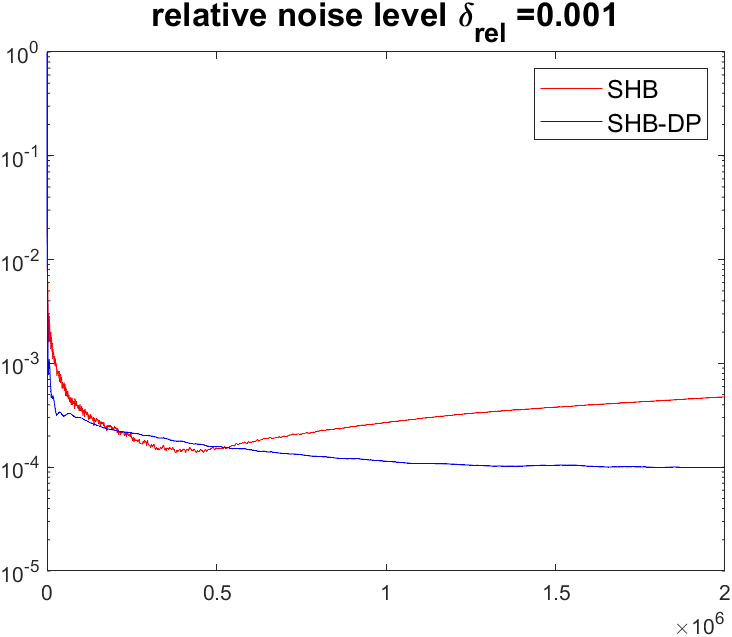}
\caption{Illustration of the effect for $\eta_{i_n}$ chosen by (\ref{DP}) which incorporates 
the discrepancy principle.} \label{fig:SHB2}
\end{figure}

It is interesting to compare the performance of Algorithm \ref{alg:SHB} with the stochastic gradient method
\begin{align}\label{SGD}
x_{n+1}^\d = x_n^\d - \eta_{i_n} A_{i_n}^*(A_{i_n} x_n^\d - y_{i_n}^\d)
\end{align}
which has been studied recently in \cite{JL2019,JLZ2023}, where $i_n\in \{1, \cdots, p\}$ is 
randomly selected via the uniform distribution. This method is a special case of (\ref{SHBM}) 
with $\a_n = 1$ and $\beta_n = 0$. We execute both Algorithm \ref{alg:SHB} and (\ref{SGD}) 
with the same $\eta_{i_n} = 0.6/\|A_{i_n}\|^2$. In the first row of Figure \ref{fig:SHB3} we 
plot the average of relative errors for $100$ independent runs and in the second row we plot 
the relative error for an individual run. From these plots, we can observe that the both
methods have comparable performance. Furthermore, although the both method exhibit semi-convergence, 
the appearance of the momentum term enables Algorithm \ref{alg:SHB} to produce less oscillatory 
iterative sequence and helps slow down the divergence of iterates.

\begin{figure}[htpb]
\centering
\includegraphics[width = 0.32\textwidth]{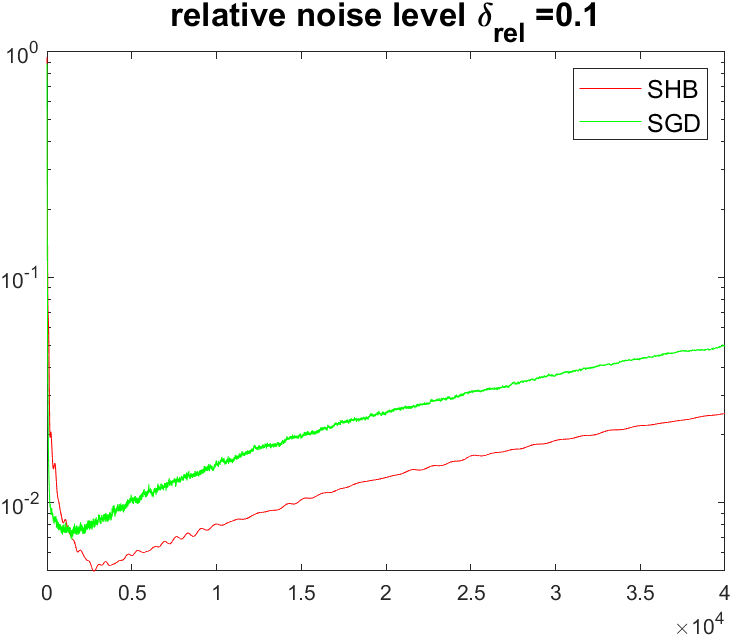}
\includegraphics[width = 0.32\textwidth]{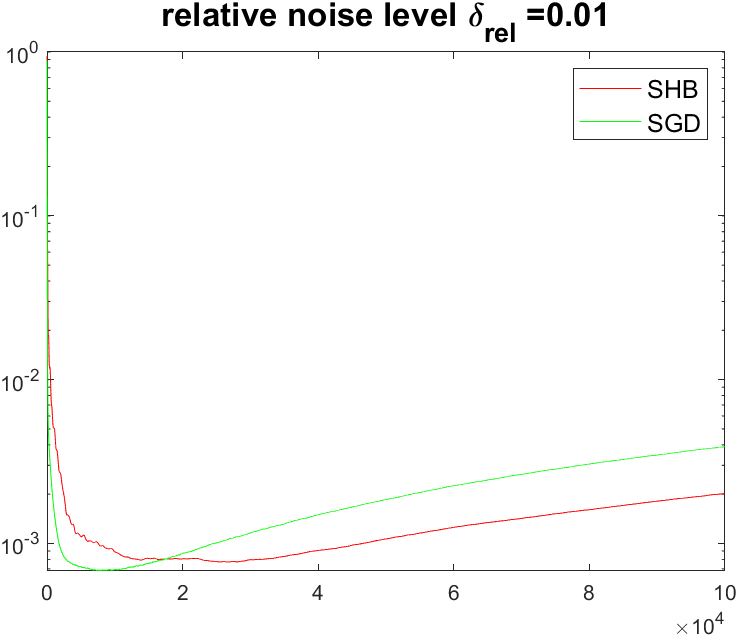}
\includegraphics[width = 0.32\textwidth]{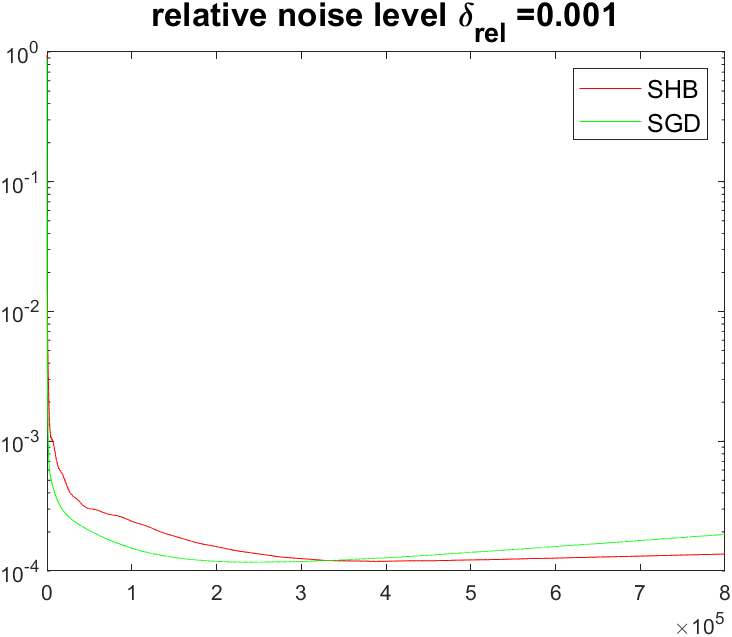}
\includegraphics[width = 0.32\textwidth]{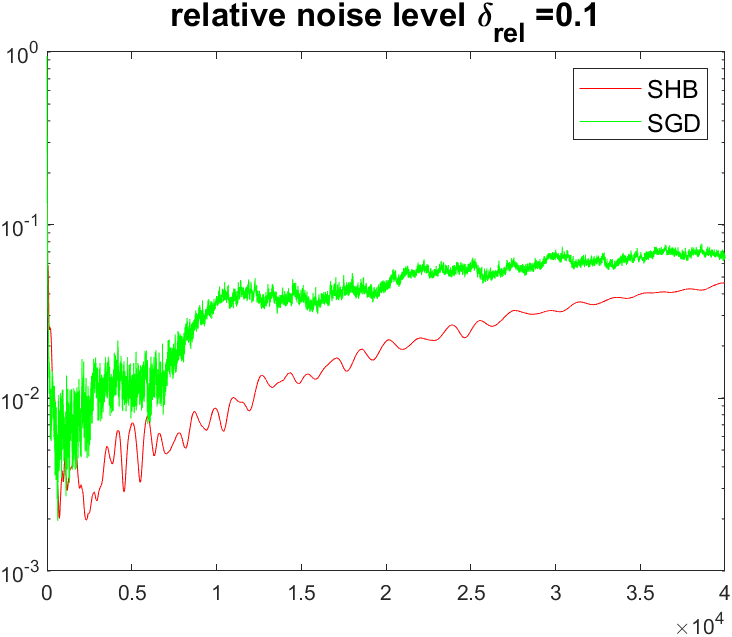}
\includegraphics[width = 0.32\textwidth]{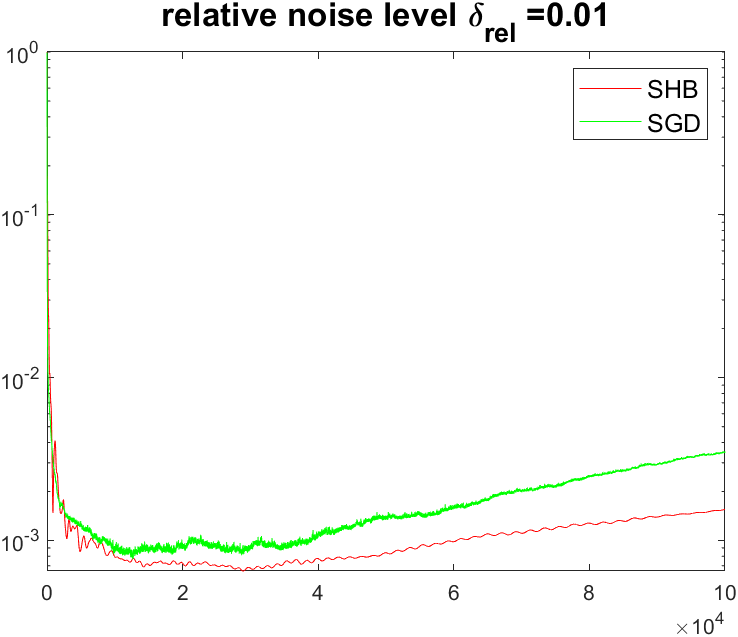}
\includegraphics[width = 0.32\textwidth]{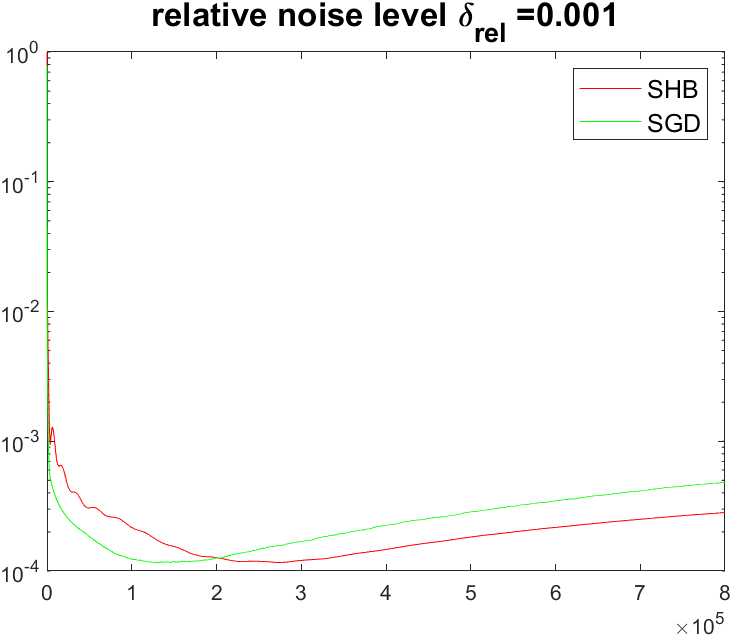}
\caption{Comparison between Algorithm \ref{alg:SHB} and the stochastic gradient descent 
method (\ref{SGD}).} \label{fig:SHB3}
\end{figure}

\end{example}

\begin{example}
We consider again the linear system (\ref{IES}) and assume the sought solution $x^\dag$ is a 
probability density function, i.e. $x^\dag \ge 0$ a.e. on $[a,b]$ and $\int_a^b x^\dag = 1$. 
Since the kernel function $\kappa$ satisfies $\kappa \in C([a,b]\times [a,b])$, each $A_i$ is 
a bounded linear operator from $L^1[a, b]$ to $\RR$. We determine such a solution by considering 
(\ref{SHB.B1}) with
\begin{equation}\label{entropy}
{\mathcal R}(x) := f(x) + \iota_{\Delta}(x),
\end{equation}
where $f$ is the negative of the Boltzmann-Shannon entropy, i.e.,
\begin{equation*}
f(x) = \begin{cases}
\int_a^b x(t)\log x(t) \; \mathrm{d}t & \textnormal{if } x \in L_{+}^1[a,b] 
\mbox{ and } x\log x \in L^1[a,b],\\
+\infty & \textnormal{otherwise,}
\end{cases}
\end{equation*}
and $\iota_{\Delta}$ denotes the indicator function of 
\begin{equation*}
\Delta := \left\{x \in L_{+}^1[a,b]: \int_a^b x(t)\;\mathrm{d}t = 1\right\}.
\end{equation*}
Here $L_{+}^{1}[a,b] := \left\{x \in L^1[a,b]: x \ge 0 \textnormal { a.e. on } [a,b]\right\}$. 
According to \cite{B1991,E1993,Jin2022}, $\mathcal{R}$ is proper, lower semi-continuous and 
$\mu$-strongly convex with $\mu = 1/2$. Furthermore, one can show that the updating formula for 
$x_n^\d$ in Algorithm \ref{alg:SHB2} is given by (see \cite{JLZ2023})
\begin{equation*}
x_n^\d = \nabla \mathcal{R}^{*}(\xi_n^\d) 
= \arg\min_{x \in L^{1}[a,b]}\left\{\R(x) - \l \xi_n^\d, x\r\right\} 
= \frac{e^{\xi_n^\d}}{\int_a^b e^{\xi_n^\d(t)}\;\mathrm{d}t}.
\end{equation*}

\begin{figure}[htpb]
\centering
\includegraphics[width = 0.32\textwidth]{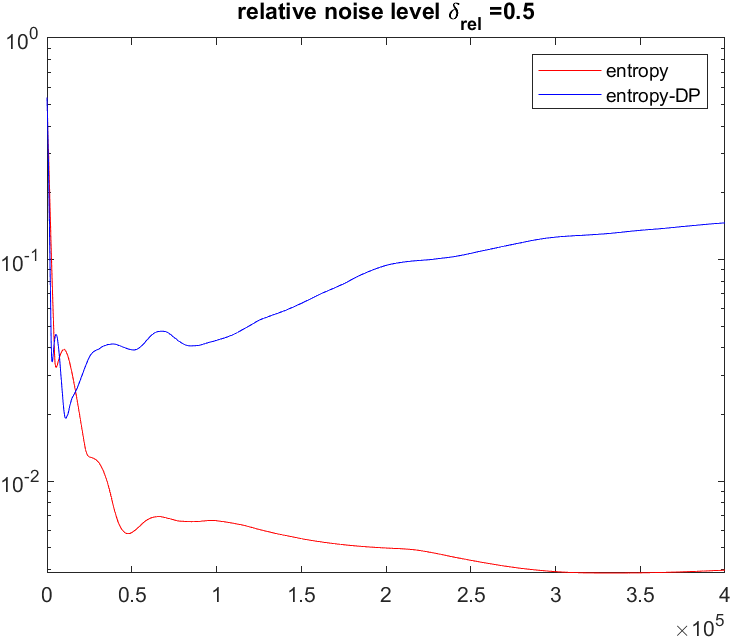}
\includegraphics[width = 0.32\textwidth]{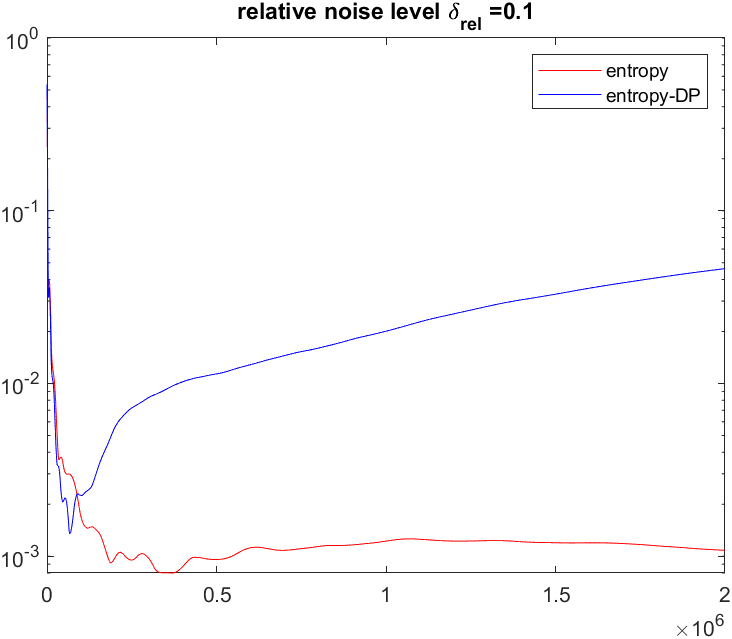}
\includegraphics[width = 0.32\textwidth]{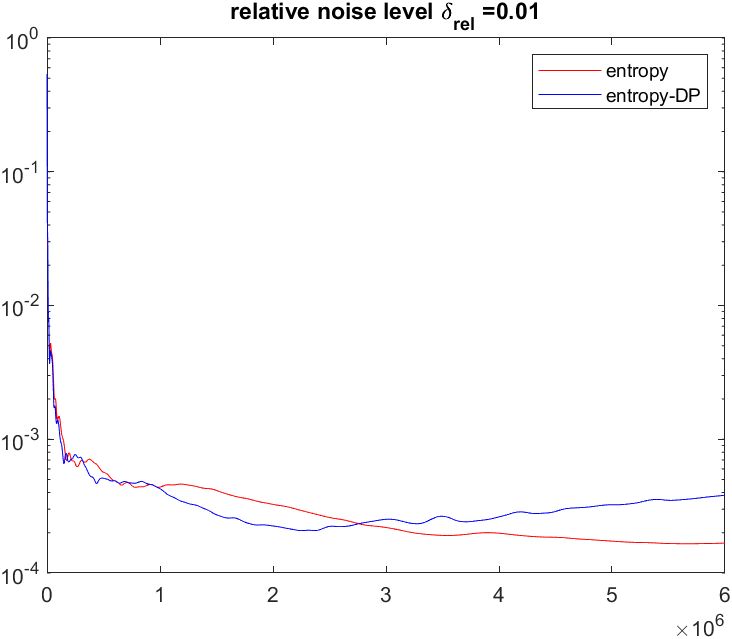}
\caption{Relative reconstruction error by Algorithm \ref{alg:SHB2}. \texttt{entropy}: 
$\eta_{i_n} = 0.98/\|A\|^2$; \texttt{entropy-DP}: $\eta_{i_n}$ is chosen by (\ref{DP}
with $\mu_0 = 0.98$.} \label{fig:SHB4}
\end{figure}    

For numerical simulations we consider the linear system (\ref{IES}) 
with $[a, b] = [0,1]$, $p = 1000$ and $\kappa(s, t) := 4 e^{-(s-t)^2/0.0064}$. We 
assume the sought solution is 
$$
x^\dag(t):= c \left(e^{-60(t-0.3)^2} + 0.3 e^{-40(t-0.8)^2}\right),
$$
where $c>0$ is a constant to ensure that $\int_0^1 x^\dag (t) dt = 1$ so that $x^\dag$ 
is a probability density function. Instead of the exact data $y_i:=A_i x^\dag$ we use 
the noisy data $y_i^\d$ of the form (\ref{nd}); we then reconstruct the sought solution $x^\dag$
using these noisy data in Algorithm \ref{alg:SHB2} with $\R$ given by (\ref{entropy}) and 
$\eta_{i_n} = 0.98/\|A\|^2$; the integrals involved in the method are approximated by the 
trapezoidal rule based on the partition of $[0,1]$ into $p-1$ subintervals of equal length. We 
execute the algorithm with noisy data for three distinct relative noise levels
$\delta_{rel} = 0.5$, $0.1$ and $0.01$ and plot in Figure \ref{fig:SHB4} the $L^1$ relative errors 
$\|x_n^\d - x^\dag\|^2_{L^1}/\|x^\dag\|^2_{L^1}$ for an individual run for each noise 
level; see the red plots labelled as ``\texttt{entropy}". These plots demonstrate that 
Algorithm \ref{alg:SHB2} with constant $\eta_{i_n}$ possesses the semi-convergence property, 
i.e., the iterates approach the sought solution at the beginning stage and then diverge after 
a critical number of iterations. 

To reduce the effect of semi-convergence, we next consider Algorithm \ref{alg:SHB2} with 
$\eta_{i_n}$ chosen by (\ref{DP}) with $\tau = 1$ and $\mu_0 = 0.98$ which uses the knowledge of 
$\d_{i_n}:=\d_{rel} \|y\|_\infty$ for $i =1, \cdots, p$ and incorporates the spirit of the 
discrepancy principle. We execute the corresponding algorithm using the same noisy data as 
above. In Figure \ref{fig:SHB4} we plot the $L^1$ 
relative reconstruction error $\|x_n^{\delta} - x^\dag\|_{L^1}^2/\|x^\dag\|_{L^1}^2$, see the blue 
plots labelled as ``\texttt{entropy-DP}". The plots in Figure \ref{fig:SHB4} demonstrate clearly 
that using $\eta_{i_n}$ chosen by \eqref{DP} can significantly relieve the method from 
semi-convergence.
\end{example}

\section{\bf Conclusion}

Heavy ball method was proposed by Polyak in 1964 to accelerate the gradient method by adding a 
momentum term. This method and its stochastic variant have received much attention in optimization
community in recent years due to the development of machine learning and the appearance of large 
scale problems. In this paper we considered a stochastic heavy ball method for solving linear 
ill-posed inverse problems. With suitable choices of the step-sizes and the momentum coefficients, 
we established the regularization property of the method and derived the rate of convergence under 
a benchmark source condition on the sought solution. There are many challenging and interesting 
questions that deserve for further study. 
These include the choices of $\a_n$ and $\beta_n$ and the design of stopping rules. Note that 
in this paper we only considered the method (\ref{SHBM}) with $\a_n$ and $\beta_n$ chosen 
by (\ref{ab}). There must exist other possible choices of $\a_n$ and $\beta_n$ that might 
yield faster convergence. It is also of interest to consider adaptive choices of $\a_n$ 
and $\beta_n$ to produce fast convergent method. Furthermore, in our paper, we considered 
terminating the method (\ref{SHBM}) by {\it a priori} stopping rules which usually require 
extra information, such as source information, of the sought solution. These extra information 
are hardly available in practical applications. How to design efficient {\it a posteriori} stopping 
rules for the method (\ref{SHBM}) is therefore significantly important. 




\bibliographystyle{amsplain}

\end{document}